\documentclass[12pt, A4paper]{article}
\usepackage{amsmath,amssymb,amscd,amsthm}
\usepackage[matrix,arrow,curve]{xy}
\CompileMatrices

\newtheorem{theorem}{Theorem}[section]
\newtheorem{predl}[theorem]{Proposition}
\newtheorem{lemma}[theorem]{Lemma}

\newtheorem{definition}[theorem]{Definition}

\newcommand{\dash}{~---\ }
\newcommand{\C}{\mathbb C}

\newcommand{\Z}{\mathbb Z}
\newcommand{\G}{\mathbb G}
\renewcommand{\P}{\mathbb P}
\renewcommand{\AA}{\mathcal A}
\newcommand{\BB}{\mathcal B}
\renewcommand{\SS}{\mathcal S}
\newcommand{\D}{\mathcal D}
\newcommand{\EE}{\mathcal E}
\newcommand{\FF}{\mathcal F}
\newcommand{\GG}{\mathcal G}

\newcommand{\LL}{\mathcal L}
\newcommand{\TT}{\mathcal T}
\renewcommand{\O}{\mathcal O}

\newcommand{\Spec}{\mathbf{Spec}}
\renewcommand{\k}{\mathsf k}

\newcommand{\ra}{\mathbin{\rightarrow}}

\newcommand{\xra}{\xrightarrow}

\renewcommand{\le}{\leqslant}
\renewcommand{\ge}{\geqslant}
\renewcommand{\~}{\tilde }

\DeclareMathOperator{\Hom}{\textup{Hom}}
\DeclareMathOperator{\HHom}{\mathcal{H}\it{om}}
\DeclareMathOperator{\Ext}{\textup{Ext}}
\DeclareMathOperator{\Gr}{\mathrm{Gr}}
\DeclareMathOperator{\Ind}{\mathrm{Coind}}
\DeclareMathOperator{\Res}{\mathrm{Res}}
\DeclareMathOperator{\Pic}{\mathrm{Pic}}
\DeclareMathOperator{\End}{\mathrm{End}}
\def\a{\alpha}
\def\b{\beta}
\def\g{\gamma}

\def\s{\sigma}
\def\lam{\lambda}
\def\Si{\Sigma}
  
\begin{document}

\author{A.\,Elagin}
\title{Semiorthogonal decompositions of derived categories 
of equivariant coherent sheaves
\thanks{
This work was partially supported by a CRDF grant RUM1-2661-MO-05.}
}
\date{}
\maketitle

\begin{abstract}                                                      
 Let $X$ be an algebraic variety with an action of an algebraic group~$G$.
Suppose $X$ has a full exceptional collection of sheaves, and 
these sheaves are invariant under the action of the group.
We construct a semiorthogonal decomposition of bounded derived category
of $G$-equivariant coherent sheaves on $X$ into components,
equivalent to derived categories of twisted representations of the group.
If the group is finite or reductive over the algebraically closed field
of zero characteristic, this gives a full exceptional collection in the
derived equivariant category. We apply our results to particular varieties
such    as projective spaces, quadrics, Grassmanians and Del Pezzo surfaces.
\end{abstract}

\section*{Introduction.}
Let $X$ be an algebraic variety over the field $\k$ with an action of
an algebraic group $G$.
In this paper we investigate $\D^b(coh^G(X))$\dash 
the derived category of $G$-equivariant
coherent sheaves on~$X$.
We prove that under some conditions the category 
$\D^b(coh^G(X))$ admits a semiorthogonal decomposition into
subcategories, equivalent to derived categories of twisted representations
of the group.
If the group has a semisimple category of representations, 
this decomposition gives a full exceptional collection, consisting of blocks.

The most simple variant of our result is the following statement 
(theorem~\ref{th1a}).
Suppose that the derived category $\D^b(coh(X))$ has a full
exceptional collection of equivariant sheaves 
$(\EE_1,\dots,\EE_n)$. Then the category $\D^b(coh^G(X))$ admits
a semiorthogonal decomposition
$$\langle\EE_1\otimes\D^b(Repr(G)),\dots,\EE_n\otimes\D^b(Repr(G))\rangle,$$
where subcategories $\EE_i\otimes\D^b(Repr(G))$ are equivalent to
derived categories of representations of the group $G$. 
As a corollary (theorem~\ref{th1}), we deduce that 
if the group $G$ is finite, $char(\k)$ is coprime with $|G|$ and
 $\k=\bar \k$, then $\D^b(coh^G(X))$ has a full exceptional collection
of sheaves
$$
\left(
\begin{array}{clclc}
\EE_1\otimes V_1 && \EE_2\otimes V_1 && \EE_n\otimes V_1 \\
\vdots & , & \vdots & ,   \dots , & \vdots  \\
\EE_1\otimes V_m && \EE_2\otimes V_m && \EE_n\otimes V_m
\end{array}
\right),
$$
where $V_1,\dots,V_m$ denote all irreducible representations of $G$ over $\k$.

These results can be viewed as a variant of the below theorem
(\cite[theorem 3.1]{Sa}).
Let $X\xra{p} S$ be a flat proper morphism of smooth schemes.
Suppose there are sheaves 
$E_1,\dots,E_n$ on $X$, such that for any fiber $X_s$ the collection 
$(E_1|_{X_s},\dots,E_n|_{X_s})$ is exceptional and full. Then we have 
a semiorthogonal decomposition 
$$\D^b(X)=\langle E_1\otimes p^*\D^b(S),\dots,E_n\otimes p^*\D^b(S)\rangle.$$
In our case we have a morphism of stacks $X/\!\!/G\ra pt/\!\!/G$. 
Here the base is a quotient stack $pt/\!\!/G$, sheaves on this stack are
representations of the group $G$, and sheaves on the total space 
$X/\!\!/G$ are exactly equivariant sheaves on~$X$. Consider a cartesian
square 
$$\xymatrix{
X \ar[r] \ar[d] & X/\!\!/G \ar[d]\\
pt \ar[r] & pt/\!\!/G.
}$$
The morphism $X\ra X/\!\!/G$ can be viewed as an embedding of the
fiber over the unique closed point  $pt\ra pt/\!\!/G$ on the base.
The roles of $E_i$'s are played by equivariant sheaves that form a
full exceptional collection after forgetting of the group action, i.e.
after restriction to a fiber.

An analogue of theorem~\ref{th1a} 
holds under weaker assumptions on the exceptional
collection $(\EE_1,\dots,\EE_n)$. In fact, sheaves $\EE_i$ need to be
just invariant, not necessarily equivariant.
To treat the case of invariant sheaves, one has to consider
not representations but twisted representations.
Our usage of twisted sheaves is parallel to that of 
M.\,Bernardara. In~\cite{Be} he applied twisted sheaves 
to describe semiorthogonal decompositions
for relative  Brauer Severi schemes, see section~\ref{proj}.
We introduce a notion of 
"cocycle" on an algebraic group. This notion of cocycle 
is
related with classification of central extensions of a group by $\G_m$
and generalizes
$2$-cocycles of abstract groups with coefficients in $\k^*$.
Our main result
(theorem~\ref{th3a}) gives a semiorthogonal decomposition
of the category $\D^b(coh^G(X))$ under the following conditions:
the category $\D^b(coh(X))$ has a full exceptional collection,
consisting of blocks of sheaves, such that the action of $G$
permutes sheaves inside each block.

In the second part of the paper we apply the developed theory
to specific varieties\dash projective spaces, quadrics, 
Grassmann varieties and Del Pezzo surfaces of degree $d\ge 5$.
        
The paper consists of three sections. 
In the first section we introduce necessary definitions and notations, 
and develop theory of cocycles on groups, twisted representations and twisted
equivariant sheaves. The main theorems are in the second section. 
We found it reasonable to consider a 
special case of finite groups
before the general case. The proof of theorem~\ref{th1} for
finite groups 
allows to demonstrate all necessary ideas but is more straightforward
then the proof in the general context of algebraic groups.
It the third section we apply the theory to particular varieties.

Author thanks D.\,Orlov and A.\,Kuznetsov for useful and 
stimulating discussions and constant attention to the work.

\section{Preliminaries: cocycles, twisted representations and
twisted sheaves.}
\label{prelim}

We will work under the following agreements.
A variety will mean a smooth algebraic variety over an arbitrary field~$\k$,
a sheaf on the variety $X$ will mean a coherent sheaf of 
$\O_X$-modules. A group will denote 
an algebraic group over $\k$ except special cases
(for example, in section~\ref{fingrp} groups will be finite).
All derived categories we consider are bounded, say 
$\D(X)$ will be used for $\D^b(coh(X))$.
Below we introduce and discuss necessary notions and their basic properties. 

\paragraph{Equivariant sheaves.}
Suppose $X$ is an algebraic variety with the action of a finite 
group $G$.
An \emph{equivariant G-sheaf} on $X$ (or simply a {$G$-sheaf})
is a sheaf $F$ on $X$ together with isomorphisms
$\theta_g\colon  F\ra g^*F$ for any  $g\in G$ such that
the diagram 
$$\xymatrix{
F\ar[r]^-{\theta_h}\ar@/_/[rrd]_{\theta_{gh}} &
h^*F\ar[r]^-{h^*\theta_g}  &
h^*g^*F\ar@{=}[d] \\
&&(gh)^*F
}$$
is commutative for any pair $g,h\in G$.
A \emph{morphism} of equivariant sheaves 
is a morphism $f\colon F_1\ra F_2$ in the category of sheaves,
compatible with isomorphisms $\theta_1$ and $\theta_2$, i.\,e.\, 
such that 
$\theta_{2,g}\circ f=g^*f\circ \theta_{1,g}$ for all $g\in G$.

In the case of a variety $X$ with an action of an algebraic group $G$
the definitions are slightly different.    
Let $\mu\colon
G\times G\ra G$ be the multiplication in $G$ and 
$a\colon G\times X\ra X$ be the action morphism.
We denote projection of direct products $G\times G, G\times X$ or 
$G\times G\times X$
onto $i$-th factor by $p_i$ and projection of $G\times G\times X$
or $G\times G\times G$  onto the product of first two (or last two) factors
by $p_{12}$ (or $p_{23}$) respectively.
By definition,
a \emph{$G$-sheaf} on $X$ is a sheaf $F$ on $X$ together
with an isomorphism $\theta\colon p_2^* F\ra a^* F$ of sheaves on $G\times X$,
satisfying the associativity condition:
the diagram 
$$\xymatrix{
p_3^*F\ar[r]^-{p_{23}^*\theta}\ar@/_/[rrrd]_{(\mu\times Id)^*\theta} &
(ap_{23})^*F\ar[rr]^-{(Id\times a)^*\theta} & &
(a(Id\times a))^*F\ar@{=}[d] \\
&&&(a(\mu\times Id))^*F
}$$
of sheaves on $G\times G\times X$ is commutative.

Morphisms in the category of $G$-sheaves on $X$ are defined as
morphisms in the category of sheaves compatible with $\theta$-s.
                                                                 
A group of morphisms in the category of $G$-sheaves between 
$\FF$ and $\GG$ will be denoted by $\Hom_G(\FF,\GG)$. 
Note that there is a natural action of $G$ on the 
space $\Hom(\FF,\GG)$ 
of morphisms in the usual category of sheaves, and we have
$\Hom_G(\FF,\GG)=(\Hom(\FF,\GG))^G$.

Equivariant coherent $G$-sheaves on $X$ form an abelian category, 
which is denoted as $coh^G(X)$. We will write $\D^G(X)$ or $\D^b(coh^G(X))$
for the bounded derives category of $coh^G(X)$.

In the case when $X$ is a point and $F$ is a vector space, 
the above definition gives a notion
of an (algebraic) representation of the group $G$. Of course,
this definition is equivalent to the standard one: a representation
of $G$ in the space $V$ is a homomorphism $G\ra GL(V)$. 
The category of finite dimensional representations of $G$ 
will be denoted as $Repr(G)$.

\paragraph{Cocycles, twisted representations and sheaves: case of finite groups}

Suppose $G$ is a finite group, $\k$ is a field, and
$\a$ is a 2-cocycle of $G$ with coefficients in $\k^*$. 
Basic definitions and facts on cohomologies of groups can be found, 
for instance,
in~\cite{Br}.

Define an \emph{$\a$-representation} of the group $G$ in the vector space
$V$ over $\k$ as a map $R\colon G\ra GL(V)$ such that $R(g)R(h)=\a(g,h)R(gh)$
for any $g,h\in G$. 
Define a \emph{twisted group algebra(?)} as follows. Let 
$\k_{\a}[G]$ be equal to $\oplus_{g\in G}\k\cdot g$
as a vector space, and define multiplication on the basis by
$g\cdot  h=\a(g,h)(gh)$.

Evidently, the categories of $\a$-representations of the group $G$ over $\k$
and of representations of the algebra $\k_{\a}[G]$ over $\k$ are equivalent,
we will denote both by $Repr(G,\a)$.


\begin{predl}
\label{twistedalg}
Algebra $\k_{\a}[G]$ is associative and possesses a unity element. 
It is semisimple if $char(\k)$ doesn't divide the order of $G$.            
Up to an isomorphism, the algebra $\k_{\a}[G]$ depends only of the class
$[\a]\in H^2(G,\k^*)$ in the 2nd cohomologies of $G$.
\end{predl}

\begin{proof}
The element $e/\a(e,e)$ is an identity. The associativity follows
directly from the cocycle condition. 
To show that $\k_{\a}[G]$ is semisimple it suffices to check that
any invariant subspace 
$U\subset V$ in an $\a$-representation $R \colon G \ra GL(V)$
has a complementary invariant subspace.
Choose any projector  $p\colon V \ra U$. Let 
$$p'=\frac 1{|G|}\sum_{g\in G}R(g)pR(g)^{-1}.$$
For all $h \in G$
we have
\begin{multline*}
R(h)p'R(h)^{-1}=\frac 1{|G|}\sum_{g\in G}R(h)R(g)pR(g)^{-1}R(h)^{-1}=\\
=\frac 1{|G|}\sum_{g\in G}\a(h,g)R(hg)pR(hg)^{-1}\a(h,g)^{-1}=p',
\end{multline*}
so $p'$ is an equivariant projector onto $U$.
Now one can take $U^{\perp}=ker\, p'$ as a required complementary subspace.

Suppose  a cocycle $\a\b$ is obtained from $\a$ by multiplication by a cochain
 $\b=\partial\g$, where 
$\b(g,h)=\g(g)\g(h)\g(gh)^{-1}$. Then an algebra isomorphism $\k_{\a}[G]\ra
\k_{\a\b}[G]$ can be given by the mapping $g\mapsto g/\g(g)$.
\end{proof}

Let $X$ be an algebraic variety over $\k$, let $G$ be a finite group
acting on $X$, let $\a$ be a 2-cocycle of $G$ with coefficients in $\k^*$.

By definition, an \emph{$\a$-$G$-equivariant sheaf} on $X$ is a coherent
sheaf $F$ together with isomorphisms $\theta_g\colon  F\ra g^*F$ for
all $g\in G$ such that $\a(g,h)\theta_{gh}=h^*(\theta_g)\circ \theta_h$ for
any pair $g,h\in G$. In the  case of trivial cocycle $\a(g,h)=1$
this gives us a common definition of a $G$-sheaf. We will often call
$\a$-$G$-equivariant sheaves simply \emph{$\a$-sheaves}.
Some properties of $\a$-$G$-sheaves are presented in the following 
proposition.

\begin{predl}
\label{asheves}
Suppose $G$ is a finite group acting on a variety $X$, and
$\a$ is a cocycle of the group $G$. Then 
 $\a$-$G$-sheaves on $X$ form an abelian category.
Let $\a$ and $\b$ be 2-cocycles of $G$, let $\FF$ and  $\GG$ be $\a$- and
$\b$-sheaves on $X$, let $U$ and $V$ be $\a$ and $\b$-representations of 
the group $G$. Then:
\begin{itemize}
\item $U\otimes V$ is an $\a\b$-representation of the group  $G$,
\item $U^*$ is an $\a^{-1}$-representation,  
\item $\Hom(U,V)$ is an $\a^{-1}\b$-representation,
\item $\FF\otimes\GG$
is an $\a\b$-sheaf on $X$,
\item $\FF^*$ is an $\a^{-1}$-sheaf,
\item $\HHom(\FF,\GG)$ is an 
$\a^{-1}\b$-sheaf,
\item $\mathcal O_X\otimes V$
is a $\b$-sheaf,
\item $\FF\otimes V$
is an $\a\b$-sheaf,
\item there is a canonical $\a$-representation of $G$ in the space
 $\Gamma(X,\FF)$,
\item there is a canonical $\a^{-1}\b$-representation of $G$ in the 
space $\Hom(\FF,\GG)$.
\end{itemize}
\end{predl}
The proof is omitted.

\medskip 
The following proposition states that, in some sense, 
twisted representations can be studied in terms of 
usual (nontwisted) representations.

\begin{predl}
\label{aboutH}
Suppose $G$ is a group of order $n$, $\k$ is an algebraically closed field of characteristic zero.
Let $\mu_d\subset \k^*$ be the subgroup,
formed by roots of unity of degree $d$.

1. Suppose that $\a\in Z^2(G,\k^*)$ is a cocycle and there exists a $d$-dimensional $\a$-representation $V$
of the group $G$. Then $d\cdot[\a]=0$ in $H^2(G,\k^*)$ and 
$[\a]=[\a']$ for some cocycle $\a'\in Z^2(G,\mu_d)$. 

2. Group $H^2(G,\k^*)$ is of $n$-torsion and canonical map $H^2(G,\mu_n)\ra H^2(G,\k^*)$ is epimorphic.

3. Consider the central extension $\bar G$ of the group $G$ by $\mu_d$, given by some cocycle $\a\in Z^2(G,\mu_d)$. 
Denote by 
$Repr_{(i)}(\bar G)$ a full subcategory in $Repr(\bar G)$,
formed by representations 
$R$, such that $R(\xi)=\xi^i\cdot Id$  for $\xi\in\mu_d$.
Then we have 
$$
Repr(\bar G)\cong\bigoplus_{i=0}^{d-1}Repr_{(i)}(\bar G)
\quad\text{and}\quad
Repr_{(i)}(\bar G)\cong Repr(G,\a^i).
$$
\end{predl}

\begin{proof}
1.
Consider the determinant representation $\Lambda^dV$ of $V$. It is a one-dimensional 
$\a^d$-representation of $G$, thus we get a map  
$R\colon G\ra \k^*$. Since $R(g)R(h)=\a(g,h)^dR(gh)$, we have 
$\a^d=\partial R$.
I.\,e., class $[\a]$ is a $d$-torsion.
Now consider the exact sequence of coefficients
$$0\ra\mu_d\ra \k^*\xra{d} \k^*\ra 0.$$
The following fragment in the long exact sequence in cohomologies 
$$H^2(G,\mu_d)\ra H^2(G,\k^*)\xra{d} H^2(G,\k^*)$$
allows us to find a proper element $[\a']\in H^2(G,\mu_d)$.

2. This follows from 1. Indeed, the group $G$ has a regular representation, which
is an $n$-dimensional $\a$-representation.

3. Let's recall the relation between cocycles and 
central extensions of groups 
(see \cite[chapter IV, \S 3]{Br}). Suppose there is an exact triple of groups 
$$1\ra \mu_d\ra \bar G\ra G\ra 1,$$
where $\mu_d$ is a central subgroup, and for any $g\in G$ one can 
choose an element $\bar g\in\bar G$, mapping into $g$, such that 
\begin{equation}
\label{form555}
\overline{\vphantom{h}g}\cdot\overline h=\overline{gh}\cdot\a(g,h)
\end{equation}
for any $g,h\in G$. Then $\bar G$ is called a central
extension of $G$ by $\mu_d$, given by the cocycle $\a$.

The decomposition 
$$Repr({\bar G})=\bigoplus\limits_{i=0}^{d-1}Repr_{(i)}(\bar G)$$
over characters of the central subgroup  $\mu_d\cong \Z/d\Z$ in $\bar G$
is a standard fact from representation theory.
We will call the representations in $Repr_{(i)}(\bar G)$ \emph{the representations of weight~$i$}.
Assume $R\colon \bar G\ra GL(U)$ is a representation of the weight $i$.
Consider the map
$G\ra GL(U)$, sending an element $g\in G$ to the operator $R(\bar g)$ (here
$\bar g$ is the preimage of $g$ in $\bar G$ fixed above).
Formula~(\ref{form555}) implies that this map is an $\a^i$-representation of the group $G$.
One can check that thus we obtain an equivalence between the categories $Repr_{(i)}(\bar G)$ and $Repr(G,\a^i)$.
\end{proof}

{\bf Remark.} In fact, the second statement is a particular case of
the following result about cohomologies of groups: 
$n\cdot H^i(G,M)=0$ for $i>0$ 
for any finite group $G$ of order $n$ and $G$-module $M$.

\paragraph{Cocycles, twisted representations and sheaves: case of algebraic groups.}

In section~\ref{alggrp} we'll have to consider 
"twisted representations of $G$"
and  "twisted $G$-sheaves" for an algebraic group $G$. 
Thus we need to develop the above theory of cocycles  in the
context of algebraic groups.
Roughly speaking, coefficients $\a(g,h)$ that form the cocycle need to depend algebraically
on $g$ and $h$. But it turns that the naive definition like "$\a(g,h)$
is a regular non-zero function on $G\times G$"  doesn't work.

Below we give reasonable definitions of objects, arising in 
section~\ref{alggrp}.

\begin{definition}
\label{defcoc}
Suppose $G$ is an algebraic group with the multiplication map
$\mu\colon G\times G\ra G$. Suppose $G$ acts on an algebraic variety $X$ and 
$a\colon G\times X\ra X$ is the action map.
A \emph{cocycle} on $G$ 
is a pair
$(\LL,\a)$, where $\LL$ is a linear bundle on $G$ and
$\a\colon p_1^*\LL\otimes p_2^*\LL\ra\mu^*\LL$
is an isomorphism of bundles on $G\times G$, satisfying the associativity
condition:
the isomorphisms 
$$(Id \times \mu)^*\a\circ (Id\otimes p_{23}^*\a) \quad\text{and}\quad
(\mu \times Id)^*\a\circ (p_{12}^*\a\otimes Id)
$$
of bundles
$p_1^*\LL\otimes p_2^*\LL\otimes p_3^*\LL$ and $(\mu(Id\times \mu))^*\LL$
on $G\times G\times G$ are equal.
\end{definition}

{\bf Remark.}
Compairing with the case of finite groups, the pair $(\LL,\a)$ 
is a generalization not of a cocylce, but of it's cohomology class.
Nevertheless, one can define twisted representations and sheaves 
starting from a cohomology class, not a cocycle.

Let $(\LL,\a)$ be a cocycle on an algebraic group $G$. Define an
\emph{$(\LL,\a)$-representation} $G$ in a vector space
$V$ as an isomorphism 
$\theta\colon \LL\otimes V\ra \O\otimes V$
between sheaves on $G\times G$ such that 
the diagram 
\begin{equation}
\label{soglas}
\xymatrix{
p_1^*\LL\otimes p_2^*\LL\otimes V \ar[rr]^{Id\otimes p_2^*\theta} \ar[d]_{\a\otimes Id} &&
p_1^*\LL\otimes V \ar[d]^{p_1^*\theta} \\
\mu^*\LL\otimes V \ar[rr]_{\mu^*\theta} && {\O} \otimes V
}
\end{equation}
is commutative.

Define a \emph{morphism} of $(\LL,\a)$-representations as a linear map
$V\ra U$, compatible with the isomorphisms $\theta_V$ and $\theta_U$.
Finite dimensional $(\LL,\a)$-representations of $G$ form an abelian
category, we will denote it by $Repr(G,\LL,\a)$.

Define an \emph{$(\LL,\a)$-$G$-equivariant sheaf} (or 
an \emph{$(\LL,\a)$-sheaf})
on $X$ as a sheaf $F$ on $X$ together with an isomorphism
$\theta\colon p_1^*\LL\otimes p_2^*F\ra a^*F$ on $G\times X$,
satisfying the following condition: 
the diagram
$$\xymatrix{
p_1^*\LL\otimes p_2^*\LL\otimes p_3^*F \ar[rr]^{Id\otimes p_{23}^*\theta} \ar[rrd]_{p_{12}^*\a\otimes Id}&&
p_1^*\LL\otimes (ap_{23})^*F \ar[rr]^{(Id\times a)^*\theta} &&
(a(Id\times a))^*F \ar@{=}[d]\\
&&(\mu p_{12})^*\LL\otimes p_3^*F \ar[rr]^{(\mu\times Id)^*\theta} &&
(a(\mu\times Id))^*F
}
$$
of sheaves on $G\times G \times X$ is commutative.

A morphism of $(\LL,\a)$-sheaves is defined as a sheaf homomorphism
$F_1\ra F_2$  on $X$, compatible with the isomorphisms 
$\theta_1$ and $\theta_2$.
As well as usual $G$-sheaves, $(\LL,\a)$-$G$-sheaves on $X$ form an 
abelian category which will be denoted as
$coh^{G,\LL,\a}(X)$. In the particular case $\LL=\O_G$,
$\a\colon \O\otimes\O\ra\O$ is a multiplication of functions, 
we get the category of (non-twisted) $G-$sheaves. 
If we take a point as $X$, then we obtain the category  of
$(\LL,\a)$-representations of the group $G$.

{\bf Example.} Suppose $X=\P(W)=\P_{\C}^{n-1}$, $G=PGL(W)$ ($G$ acts on
$X$ tautologically), and $\FF=\O_{\P(W)}(-1)$.
Denote by $\LL$ the linear bundle on $G$, associated with the principal
$\C^*$-bundle $GL(W)$ over $G$.
The composition law on $GL(W)$ defines a multiplication on sections of
the bundle $\LL$.
Thus we obtain a cocycle structure on $\LL$, denote it by  $(\LL,\a)$.
The morphism of action $GL(W)\times W\ra W$ extends to the 
isomorphism 
$$p_1^*\LL\otimes p_2^*(\O\otimes W)\ra a^*(\O\otimes W)$$
 of sheaves on $G\times \P(W)$. It defines the 
$(\LL,\a)$-structure on the bundle $\O_{\P(W)}\otimes W$. 
Note that $\O(-1)\subset \O\otimes W$ is a $G$-invariant subsheaf.
Hence, $\O_{\P(W)}(-1)$ is an $(\LL,\a)$-equivariant sheaf.

One can see that the bundle $\LL$ is not trivial: 
$G=PGL(W)=\P(\End W)\setminus\{det=0\}$, so the Picard group
$\Pic G$ is generated by $\O_{\P(\End W)}(1)$ and 
isomorphic to $\Z/n\Z$.
Since $\LL$ is a restriction of the tautologic linear bundle on 
$\P(\End W)$,  $\LL\not\cong \O$.

\medskip

There is a natural way to define a product on the set of cocycles on 
a fixed group $G$.
For two cocycles $(\LL_1,\a_1)$ and $(\LL_2,\a_2)$ 
their product is said to be a pair
$(\LL_1,\a_1)\cdot(\LL_2,\a_2)=(\LL_1\otimes\LL_2,\a_1\otimes\a_2)$,
where the isomorphism $\a_1\otimes\a_2$ is the composition
of maps
$$p_1^*(\LL_1\otimes\LL_2)\otimes p_2^*(\LL_1\otimes\LL_2)=
p_1^*\LL_1\otimes p_2^*\LL_1\otimes p_1^*\LL_2\otimes p_2^*\LL_2
\xra{\a_1\otimes\a_2} \mu^*\LL_1\otimes\mu^*\LL_2=\mu^*(\LL_1\otimes\LL_2).
$$
Clearly, the
 set of cocycles on $G$ with the above operation is a group.
The element in this group, opposite to $(\LL,\a)$, will be denoted as
$(\LL,\a)^{-1}=(\LL^*,\a^*)$.

In the following proposition we list some basic properties of
twisted reprsentations and sheaves. 

\begin{predl}
\label{lasheves}
Let $\FF$ be  an  $(\LL_1,\a_1)$-$G$-sheaf аnd 
$\GG$ be an $(\LL_2,\a_2)$-$G$-sheaf on $X$, let $U$ and $V$ be 
$(\LL_1,\a_1)$- and $(\LL_2,\a_2)$-representations of the group $G$
respectively. Then
\begin{itemize}
\item $U\otimes V$ is an $(\LL_1\otimes\LL_2,\a_1\otimes\a_2)$-representation of the group $G$,
\item $U^*$ is an $(\LL_1^*,\a_1^*)$-representation of $G$,
\item $\Hom(U,V)$ is an $(\LL_1^*\otimes\LL_2,\a_1^*\otimes\a_2)$-representation of $G$,
\item $\FF\otimes\GG$ is an $(\LL_1\otimes\LL_2,\a_1\otimes\a_2)$-sheaf on $X$,
\item $\FF^*$ is an $(\LL_1^*,\a_1^*)$-sheaf,
\item $\HHom(\FF,\GG)$ is an $(\LL_1^*\otimes\LL_2,\a_1^*\otimes\a_2)$-sheaf,
\item $\O_X\otimes V$ is an $(\LL_2,\a_2)$-sheaf,
\item $\FF\otimes V$ is an $(\LL_1\otimes\LL_2,\a_1\otimes\a_2)$-sheaf,
\item $\Gamma(X,\FF)$ is an $(\LL_1,\a_1)$-representation of $G$,
\item $\Hom(\FF,\GG)$ is an $(\LL_1^*\otimes\LL_2,\a_1^*\otimes\a_2)$- of $G$.
\end{itemize}
\end{predl}
The proofs are omitted.

\medskip

Likewise for the case of finite groups, cocycles defined above correspond to 
central extensions of groups by the group $\G_m$.
Assume $(\LL,\a)$ is a cocycle on an algebraic group~$G$. 
Denote the total space of the bundle $\LL$ minus zero section by $\~G$.
More formally, $\~G=\Spec_G\AA$, where 
$$\AA=\oplus_{n\in \Z}\LL^{\otimes n}$$
is a sheaf of $\O_G$-algebras on $G$.
Then $\~G$ is a principal $\G_m$-bundle.
Point out that the category of sheaves on $\~G$ and the category 
of sheaves of $\AA$-modules
on $G$ are equivalent.
Now we define an associative multiplication on $\~G$ using the 
isomorphism $\a$.

Indeed, we can define a product map
$$\~\mu\colon \~G\times\~G=\Spec_{G\times G}\left( p_1^*\AA\otimes p_2^*\AA\right) \ra
\Spec_G\AA=\~G$$
using the homomorphism of sheaves of algebras $\mu^*\AA\ra
p_1^*\AA\otimes p_2^*\AA$, which is generated by the isomorphism $\a^{-1}\colon
\mu^*\LL\ra p_1^*\LL\otimes p_2^*\LL$. The associativity condition for $\a$
implies associativity for operation $\~\mu$. Denote by $\pi$ the projection
of $\~G$ onto $G$. The fiber of $\pi$ over the identity in $G$ is
by definition 
$\mathrm{Spec}\,\k\times_G\~G=\mathrm{Spec}\,e^*\AA$, where
 $e\colon \mathrm{Spec}\,\k\ra G$ is the identity element in $G$.
We can choose an element $1_e\in e^*\LL\setminus 0$ such that 
$e^*\a(1_e\otimes 1_e)=1_e$ for the isomorphism
$e^*\a\colon e^*\LL\otimes e^*\LL\ra e^*\LL$. 
We will now identify $e^*\LL$ with
$\k=1_e\cdot \k$. Then the algebra 
$e^*\AA$ would be isomorphic to $\k[t,1/t]$, where $t\in e^*\LL^*, t\cdot 1_e=1$,
and $e^*\a$ would be the ordinary multiplication.
One can check that the identity element in $\~G$ 
can be defined by a homomorphism of algebras $\k[t,1/t]\ra \k$ 
such that $t\mapsto 1$.
Suppose  $i\colon G\ra G$ is an inversion map. Then
an inversion $\~\i$ for $\~G$ is defined by the isomorphism $\AA\ra i^*\AA$
of sheaves of algebras on $G$, which is generated by the the isomorphism
 $\LL\ra i^*\LL^{-1}$.
The latter in it's turn is induced by the pairing 
$$\LL\otimes i^*\LL\cong \Delta^*(p_1^*\LL\otimes p_2^*\LL)\xra{\Delta^*\a}
\Delta^*\mu^*\LL\cong e^*\LL\otimes \O=\O,$$
here $\Delta\colon G\ra G\times G$ stands for the antidiagonal embedding.

By construction, projection $\pi\colon \~G\ra G$ is a group
homomorphism, the kernel $\mathrm{Spec}\,e^*\AA$ is isomorphic
to $\G_m$ and belongs to the center of $\~G$. Thus, we get a central
extension
$$1\ra \G_m\ra \~G\xra{\pi} G\ra 1.$$

Consider the linear bundle $\~\LL=\pi^*\LL$ on  $\~G$ and
the isomorphism $\~\a=\pi^*\a\colon p_1^*\~\LL\otimes p_2^*\~\LL\ra
\~\mu^*\~\LL$.
Clearly, the pair $\pi^*(\LL,\a)=(\~\LL,\~\a)$ is a cocycle on $\~G$.
We claim that this cocycle is trivial.
Indeed, the bundle $\~\LL$ corresponds to the sheaf $\LL\otimes \AA$
of $\AA$-modules. There is
a straightforward isomorphism of  $\AA$-modules  $\LL\otimes \AA\ra \AA$,
which respects the multiplication~$\a$.         

Now consider a variety $X$ with the action of $G$. Evidently,
$X$ is also a $\~G$-variety. Any $(\LL,\a)$-$G$-equivariant sheaf
$\FF$ on $X$ is automatically equipped with the structure of an 
$(\~\LL,\~\a)$-$\~G$-sheaf.
We see that extending the group we obtain a
(non-twisted) equivariant sheaf from the twisted one.

The action of the subgroup $\G_m\subset \~G$ on $X$ is trivial.
How does it act on~$\FF$?
It turns out that this action is linear. The converse is also true: 
any $\~G$-sheaf on $X$, such that the action of $\G_m$ is linear,
admits a canonical $(\LL,\a)$-$G$-equivariant structure.
In fact, a bit more general statement holds:

\begin{predl}
\label{aboutHH}
Let $X$ be an algebraic variety, suppose that  $G$ is an algebraic
group, acting on $X$, and $(\LL,\a)$ is a cocycle on $G$. Consider the
extension
$1\ra \G_m\ra \~G \ra G\ra 1$, corresponding to this cocycle. 
Denote by 
$coh_{(r)}^{\~G}(X)$ the full subcategory in $coh^{\~G}(X)$, 
formed by sheaves $\FF$ such that the subgroup $\G_m\subset \~G$ 
acts on $\FF$ with weight $r$.
Then for any integer $r$ 
there is an equivalence
$$coh^{G,\LL^r,\a^r}(X) \cong coh_{(r)}^{\~G}(X).$$ 
\end{predl}

\begin{proof}
Let $(\LL,\a)^r=(\LL^r,\a^r)$ be a degree of the given cocycle,
and $\FF$ be an $(\LL,\a)^r$-equivariant sheaf on $X$ 
with the structure isomorphism 
$\theta\colon p_1^*\LL^r\otimes p_2^*\FF\ra a^*\FF$. 
Then $\FF$ is also an $(\~\LL,\~\a)^r$-$\~G$-equivariant sheaf.
As it was mentioned above, the cocycle 
$(\~\LL,\~\a)^r$ has a canonical trivialization, 
so $\FF$ can be considered as a $\~G$-sheaf.
We claim that the described correspondence gives an equivalence.
First we need to figure out how the subgroup $\G_m\subset \~G$ acts on 
the sheaf $\FF$.
The action of $\G_m$ on $\FF$ is determined by the isomorphism
$p_1^*\~\LL^r\otimes p_2^*\FF\ra p_2^*\FF$ of sheaves on $\G_m\times X$,
which comes from $\theta$. Or, in other terms,
it is determined by the isomorphism 
$$e^*\LL^r\otimes e^*\AA\otimes \FF\ra e^*\AA\otimes \FF$$
of sheaves of $e^*\AA\otimes \O_X$-modules on $X$. The latter isomorphism
would be identity if we use equality 
$e^*\LL^r=(1_e\cdot \k)^{\otimes r}=1_e\cdot \k$ .
By definition of $\~G$-structure on $\FF$, we need to apply
 the identification $e^*\LL^r\otimes e^*\AA\ra e^*\AA$, 
which is the iteration of the isomorphism $\LL\otimes \AA\ra\AA$, 
described above. As a result, we get a multiplication 
by $t^r\colon \k[t,1/t]\otimes \FF\ra \k[t,1/t]\otimes \FF$.
Therefore, $\G_m$ acts on the sheaf~$\FF$  with weight $r$.

Now let's check the converse: suppose $\FF$ is a  $\~G$-sheaf on $X$,
such that $\G_m\subset \~G$ acts on $\FF$ with weight $r$. We can introduce 
the structure of an $(\LL,\a)^r$-$G$-equivariant sheaf on $\FF$. Indeed,
the action of $\~G$ on $\FF$ is given by an isomorphism of 
sheaves of $p_1^*\AA$-modules on $G\times X$:
$$\~\theta\colon p_1^*\AA\otimes p_2^*\FF\ra p_1^*\AA\otimes a^*\FF.$$
The restriction of $\~\theta$ on 
$p_1^*\LL^r\otimes p_2^*\FF$ is a sum of the morphisms
$$\theta_i\colon p_1^*\LL^r\otimes p_2^*\FF\ra p_1^*\LL^i\otimes a^*\FF.$$
Since $\G_m$ acts on $\FF$ with weight $r$, all components except
$\theta_0$ vanish, $\theta=\theta_0$ and we get an isomorphism
$p_1^*\LL^r\otimes p_2^*\FF\ra a^*\FF$, which is associative.
\end{proof}

As a special case $X=pt$ of proposition~\ref{aboutHH} we get
\begin{predl}
Assume $(\LL,\a)$ is a cocycle on an algebraic group $G$,
and 
$1\ra \G_m\ra \~G \ra G\ra 1$ is a corresponding extension of groups.
Then for any integer $r$  one has an equivalence of categories
$$Repr(G,\LL^r,\a^r) \cong Repr_{(r)}(\~G),$$ 
where 
$Repr_{(r)}(\~G)$ denotes the full subcategory in $Repr(\~G)$ of 
representations $V$, such that the subgroup $\G_m\subset \~G$ 
acts on $V$ with weight $r$.
\end{predl}

{\bf Remark.}
Since the central subgroup $\G_m\subset \~G$ is reductive,
the category $Repr(\~G)$ admits the following
decomposition over characters of $\G_m$:
$$Repr(\~G)=\bigoplus\nolimits_{r\in\Z}Repr_{(r)}(\~G),$$
where the components $Repr_{(r)}(\~G)$ were defined above.

The similar decomposition exists for the category of 
$\~G$-equivariant sheaves on a variety:
$$coh^{\~G}(X)=\bigoplus\nolimits_{r\in\Z}coh^{\~G}_{(r)}(X).$$

We see that the category of $(\LL,\a)$-representations of $G$
is a full subcategory (in fact, a direct summand)
in the category of representations of some extension of $G$ by $\G_m$. 
Actually, it is enough to consider only finite extensions, like in the case 
of finite groups.

\begin{lemma}
Let  $(\LL,\a)$ be a cocycle on an algebraic group $G$.
Suppose there exists an $(\LL,\a)$-representation of the group
$G$ in a vector space $V$ of dimension $n$.
Then the cocycle $(\LL,\a)^n$ is trivial.
\end{lemma}

\begin{proof}
By proposition~\ref{lasheves}, there is an $(\LL,\a)^n$-representation of $G$
in one-dimensional space $\Lambda^nV$. That is, there is an isomorphism
$\theta$ of bundles on $G\colon \LL^n\ra \O$. 
The condition~(\ref{soglas}) implies that
$\theta$ sends structure isomorphism $\a^n$ into standard
multiplication $\O\otimes \O\ra \O$.
\end{proof}

Under the assumption of the lemma we will construct
an extension of the group $G$ by the algebraic group $\mu_n$. 
First we fix the isomorphism $\theta\colon \LL^n\ra \O$, 
trivializing the cocycle $(\LL,\a)^n$. 
Define a sheaf of $\O_G$-algebras on  $G$: set
$$\BB=\oplus_{i=0}^{n-1}\LL^{\otimes i}$$
and introduce a multiplication on $\BB$ using the isomorphism $\theta$.
Now let $\bar G$ be the relative spectrum $\Spec_G \BB$. 
Informally, $\bar G$ can be thought as a pullback of the identity section
under the map of raising into $n$-th power:
$\LL\ra\LL^n\cong\O$.
Arguing as above, one can prove the following

\begin{predl}
\label{aboutHH1}
The scheme $\bar G$ defined above is a closed subgroup in $\~G$,
it is a central extension of $G$ by $\mu_n$.
The following categories are equivalent for all $r\in \Z$:
$$coh^{G,\LL^r,\a^r}(X)\cong coh^{\bar G}_{(r)}(X)$$
and
$$Repr(G,\LL^r,\a^r)\cong Repr_{(r)}(\bar G).$$ 
One has the decompositions 
$$coh^{\bar G}(X)\cong \bigoplus_{r=0}^{n-1}coh^{\bar G}_{(r)}(X)$$
and
$$Repr(\bar G)=\bigoplus_{r=0}^{n-1}Repr_{(r)}(\bar G),$$
where notations $coh^{\bar G}_{(r)}(X)$ and 
$Repr_{(r)}(\bar G)$ are similar to the above.
\end{predl}

\paragraph{Admissible subcategories.}
According to a definition in~\cite[section 3]{Bo}, 
a subcategory in the given category is said to be \emph{right (left)
admissible}
if the embedding functor has a right (left) adjoint. A subcategory
is \emph{admissible} if it is both right and left admissible.
The following easy result will be needed for the sequel (\cite[lemma 3.1, 
theorem 3.2a]{Bo}). 
\begin{predl}
\label{bondal}
1. Suppose $\TT$ is a triangulated category, 
$\SS_1,\dots,\SS_n\subset \TT$ are right admissible tringulated subcategories.
Suppose the categories $\SS_i$ are semiorthogonal, i.\,e.\,
$\Hom(\SS_j,\SS_i)=0$ for $i<j$.
Then $\TT$ admits the semiorthogonal decomposition
$\TT=\langle\SS^{\perp},\SS_1,\dots,\SS_n\rangle$, where 
the right orthogonal 
$\SS^{\perp}$ is by definition a full subcategory in $\TT$,                  
consisting of all objects $X$ such that $\Hom(\SS_i,X)=0$ for any $i$.
Similar is true with "left" changed into "right".

2. If a subcategory in a tringulated category with finite-dimensional 
$\Hom$ spaces is generated by an 
exceptional collection, it is admissible.
\end{predl}

\section{General theorems.}
Theorems of this section are the main result of the paper.
We start with the special case of finite groups
to make the exposition more clear.

\subsection{Case of finite groups.}
\label{fingrp}

\begin{theorem}
\label{th1}
Let $G$ be a finite group, $X$ be an algebraic variety over a field $\k$
with the action of $G$. Suppose in addition that $\k=\bar \k$ and $char(\k)$ 
doesn't divide the order of $G$.
Suppose that the category $\D(X)=\D^b(coh(X))$ has 
a full exceptional collection of sheaves $(E_1, \dots, E_n)$.
Suppose each sheaf $E_i$ admits a $G$-equivariant structure;
denote corresponding $G$-sheaf by $\EE_i$.
Denote by $V_1,\dots,V_m$ all irreducible representations of $G$ over $\k$.
Then the collection, organized into blocks 
\begin{equation}
\label{coll1}
\left(
\begin{array}{clclc}
\EE_1\otimes V_1 && \EE_2\otimes V_1 && \EE_n\otimes V_1 \\
\vdots & , & \vdots & ,   \dots , & \vdots  \\
\EE_1\otimes V_m && \EE_2\otimes V_m && \EE_n\otimes V_m
\end{array}
\right) ,
\end{equation}
is a full and exceptional collection in the equivariant derived 
category
 $\D^G(X)=\D^b(coh^G(X))$.
\end{theorem}

\begin{proof}
First note that the category of representations of $G$ over $\k$ is semisimple,
hence for $G$-sheaves $\FF_1$ and  $\FF_2$ one has:
$$\Hom^i_G(\FF_1,\FF_2)=(\Hom^i(\FF_1,\FF_2))^G.$$
That is, equivariant $\Ext$ groups can be computed as invariants
of $\Ext$ groups in usual category of sheaves.

Now we check that the collection~(\ref{coll1}) is exceptional.
For $i<j$

$$\Hom^r_G(\EE_j\otimes V_p,\EE_i\otimes V_q)\subset
\Hom^r(E_j\otimes V_p,E_i\otimes V_q)=
\Hom^r(E_j,E_i)\otimes V_p^*\otimes V_q=0,$$
since the collection $(E_1, \dots, E_n)$ is exceptional. 
For $i=j$ we have

\begin{multline*}
\Hom^r_G(\EE_i\otimes V_p,\EE_i\otimes V_q)=
(\Hom^r(E_i\otimes V_p,E_i\otimes V_q))^G=
(\Hom^r(E_i,E_i)\otimes V_p^*\otimes V_q)^G=\\
=(\Hom^r(E_i,E_i)\otimes \Hom(V_p, V_q))^G=
\begin{cases}
\k&\text{for $p=q, r=0$,}\\
0&\text{otherwise.}
\end{cases}
\end{multline*}

So the collection~(\ref{coll1}) is exceptional. 
Let's show it is full.

By proposition~\ref{bondal}, it suffices to check that 
the right orthogonal to the subcategory in $\D^G(X)$,
generated by the collection~(\ref{coll1}), is zero. 
Take an object $\FF\in \D^G(X)$ such that $\FF\ne 0$.

Since the collection $(E_1, \dots, E_n)$ generates the category $\D(X)$, 
 the space $\Hom^r(\EE_i,\FF)$ is nonzero
for some $i$ and $r$. Consider the action of the group $G$ on
this space, choose any irreducible subspace $V$ in it.
Then we have
$$\Hom^r_G(\EE_i\otimes V,\FF)=\Hom_G(V,\Hom^r(\EE_i,\FF))\ne 0.$$
Therefore, $\FF$ is not right orthogonal to the subcategory, generated by 
collection~(\ref{coll1}). The theorem is proved.
\end{proof}

Theorem~\ref{th1} is a special case of theorem~\ref{th1a} proved below. 
We give here this clear and compact proof of special case to 
demonstrate how the ideas, necessary for the general case, work.
  
\medskip
Assumptions on sheaves $E_1,\dots,E_n$ in theorem~\ref{th1} 
can be weakened: the construction works for invariant sheaves, not only
equivariant.

Let, as usual, $G$ be a finite group acting on a variety $X$. 
Suppose that $E$ is an exceptional coherent sheaf on $X$ and
$g^* E\cong E$ for any $g\in G$  
(in this case we say that $E$ is invariant under the action of $G$).
We claim that $E$ can be made an $\a$-sheaf for some 2-cocycle $\a$ of
the group $G$.
Indeed, let's fix an isomorphism $\theta_g\colon E\ra g^*E$ for
every $g$. Since $E$ is exceptional,
the isomorphisms $\theta_{gh}$ and 
$h^*(\theta_g)\circ  \theta_h\colon  E\ra (gh)^*E$
differ by multiplication by a scalar. Denote this scalar by $\a(g,h)$. 
One can check that $\a(g,h)$ form a cocycle and that the cohomology class
of that cocycle doesn't depend on the choice of isomorphisms $\theta_g$. 
We see that the sheaf $E$ with the isomorphisms~$\theta_g$ 
is an $\a$-sheaf by definition.

\begin{theorem}
\label{th2}
Let $X$ be an algebraic variety over a field $\k$
with the action of a finite group $G$. Suppose in addition that $\k=\bar \k$ and 
$char(\k)$ 
doesn't divide the order of $G$.
Suppose that the category $\D(X)$ has 
a full exceptional collection of sheaves $(E_1, \dots, E_n)$.
Assume that every sheaf $E_i$ is invariant under the action of $G$, 
i.e. $g^*E_i\cong E_i$ for all $g$ и $i$.
Make the sheaf $E_i$ equivariant with respect to some cocycle $\a_i$ of
the group $G$, denote the $\a_i$-sheaf we obtain by $\EE_i$.
Let $V^{(i)}_1,\dots,V^{(i)}_{m_i}$ be all irreducible 
$\a_i^{-1}$-representations of the group  $G$ over $\k$.

Then the collection, consisting of blocks 
\begin{equation*}
\left(
\begin{array}{clclc}
\EE_1\otimes V^{(1)}_1 && \EE_2\otimes V^{(2)}_1 && \EE_n\otimes V^{(n)}_1 \\
\vdots & , & \vdots & ,   \dots , & \vdots  \\
\EE_1\otimes V^{(1)}_{m_1} && \EE_2\otimes V^{(2)}_{m_2} && \EE_n\otimes
V^{(n)}_{m_n}
\end{array}
\right) ,
\end{equation*}
is a full exceptional collection in the category $\D^G(X)$.
\end{theorem}

The proof repeats the proof of theorem~\ref{th1}.
We omit it because theorem~\ref{th2} is a corollary from 
theorem~\ref{th2a} below.
To deduce this corollary it suffices to note that the categories
of $\a_i^{-1}$-representations of $G$ are semisimple by
proposition~\ref{twistedalg}.

\medskip
For further generalizations we need a notion of coinduced $G$-sheaf.
Let $G$ be a group acting on a variety $X$, and $H$ be a subgroup of $G$.
Then the forgetful functor $\Res_H^G$ from $G$-sheaves to $H$-sheaves
has the right adjoint functor which is denoted $\Ind_H^G$. 
Given an $H$-sheaf $\FF$ on $X$, the $G$-sheaf $\Ind_H^G(\FF)$ is called
a sheaf, \emph{coinduced from the subgroup $H$}.
As a sheaf $\Ind_H^G(\FF)$ is isomorphic to 
$\bigoplus_{g\in J}g^*\FF$, where $g$ runs the set $J$ of representatives
of right cosets $H\setminus G$. See section~\ref{alggrp} for details.

\begin{theorem}
\label{th3}
Let $X$ be an algebraic variety over a field $\k$ with the action of
a finite group $G$. We assume that $\k=\bar \k$ and $char(\k)$ doesn't
divide the order of $G$.
Suppose that the category $\D(X)$ has a full exceptional collection,
consisting of blocks:

\begin{equation*}
\left(
\begin{array}{clclc}
E^{(1)}_1 && E^{(2)}_1 && E^{(n)}_1 \\
\vdots & , & \vdots & ,   \dots , & \vdots  \\
E^{(1)}_{k_1} && E^{(2)}_{k_2} && E^{(n)}_{k_n}
\end{array}
\right) ,
\end{equation*}
such that $G$ transitively permutes the sheaves inside each block. 
Let $H_i$ be the stabilizer of the sheaf $E^{(i)}_1$.
Make $E^{(i)}_1$ an equivariant sheaf with respect to some 2-cocycle
$\a_i$ of the group $H_i$, denote the $\a_i$-$H_i$-sheaf 
we obtain by $\EE^{(i)}$.
Let $V^{(i)}_1,\dots,V^{(i)}_{m_i}$ be all irreducible
$\a_i^{-1}$-representations of $H_i$ over $\k$. Then the collection, 
consisting of blocks
\begin{equation*}
\left(
\begin{array}{clclc}
\Ind_{H_1}^G(\EE^{(1)}\otimes V^{(1)}_1) &&
\Ind_{H_2}^G(\EE^{(2)}\otimes V^{(2)}_1) &&
\Ind_{H_n}^G(\EE^{(n)}\otimes V^{(n)}_1) \\
\vdots & , & \vdots & ,   \dots , & \vdots  \\
\Ind_{H_1}^G(\EE^{(1)}\otimes V^{(1)}_{m_1}) &&
\Ind_{H_2}^G(\EE^{(2)}\otimes V^{(2)}_{m_2}) &&
\Ind_{H_n}^G(\EE^{(n)}\otimes V^{(n)}_{m_n})
\end{array}
\right)
\end{equation*}
is a full exceptional collection in the category $\D^G(X)$.
\end{theorem}
This a corollary of theorem~\ref{th3a} in the following section.

\subsection{Case of algebraic groups.}
\label{alggrp}

The above theory with some modifications works for the case of arbitrary
algebraic groups. We will give generalizations of three theorems from the 
previous section.

\paragraph{The first theorem: case of exceptional collection 
of equivariant sheaves in $\D(X)$.}
Let  $X$ be an algebraic variety with the action of an algebraic group $G$.
Consider a $G$-sheaf $\EE$ on $X$ that is exceptional in the 
category $\D(X)$.
Using this sheaf, we embed the category  $Repr(G)$ of finite dimensional
representations of the group $G$ into the category $coh^G(X)$.
Define a functor 
$F_{\EE}\colon Repr(G)\ra coh^G(X)$ by sending a representation $V$ into 
the equivariant sheaf $\EE\otimes V$. The functor $F_{\EE}$ is exact,
corresponding derived functor will also be denoted by $F_{\EE}$.

\begin{lemma}
\label{lemma24}
The functors $R\Hom(\cdot,\EE)^*$ and $R\Hom(\EE,\cdot)$
from $\D^G(X)$ to $\D(Repr(G))$ are respectively 
left and right adjoint to~$F_{\EE}$.
The functor $F_{\EE}\colon \D(Repr(G))\ra \D^G(X)$ is fully faithful.
\end{lemma}

\begin{proof}
For $V\in Repr(G)$, $\FF\in coh^G(X)$ one has the canonical isomorphisms
of vector spaces
\begin{gather*}
\Hom(\FF,\EE\otimes V)\cong \Hom(\FF,\EE)\otimes V\cong
\Hom(\Hom(\FF,\EE)^*,V)\qquad\text{and}\\
\Hom(\EE\otimes V,\FF)\cong 
\Hom(V,\Hom(\EE,\FF)),
\end{gather*}
compatible with the group action.
Taking the invariants, we see that functors 
$\Hom(\cdot,\EE)^*$ and $F_{\EE}$ are adjoint on abelian categories,
the same is true for $F_{\EE}$ and $\Hom(\EE,\cdot)$.
By~\cite[lemma 15.6]{Ke}, the derived functors are also adjoint.
Now let's verify $F_{\EE}$ is fully faithful. 
Take $V$ and $U\in \D(Repr(G))$. Since the sheaf  $\EE$ is exceptional,
we get 
\begin{multline*}
\Hom_G(\EE\otimes V,\EE\otimes U)=\Hom_G(V,R\Hom(\EE,\EE\otimes U))=\\
=\Hom_G(V,R\Hom(\EE,\EE)\otimes U)=\Hom_G(V,U),
\end{multline*}
therefore $F_{\EE}$ is an embedding of categories.
\end{proof}

We will denote the image of the category $\D(Repr(G))$ under
embedding $F_{\EE}$ by $\EE\otimes \D(Repr(G))$.
Notice that the subcategory $\EE\otimes\D(Repr(G))\subset \D^G(X)$
is admissible by lemma~\ref{lemma24}.

\begin{lemma}
\label{lemma26}
Suppose $\EE_1$ and $\EE_2$ are $G$-equivariant sheaves, exceptional
in the category $\D(X)$, and suppose that $\Ext^i(\EE_1,\EE_2)=0$ for all $i$.
Then the subcategory 
$\EE_1\otimes\D(Repr(G))$ in $\D^G(X)$  is left orthogonal to
the subcategory 
$\EE_2\otimes\D(Repr(G))$.
\end{lemma}
\begin{proof}
Since $F_{\EE_1}$ and $R\Hom(\EE_1,\cdot)$ are adjoint,
we get for any $V,U\in \D(Repr(G))$:
\begin{multline*}
\Hom_G(\EE_1\otimes V,\EE_2\otimes U)=
\Hom_G(V,R\Hom(\EE_1,\EE_2\otimes U))=\\
=\Hom_G(V,R\Hom(\EE_1,\EE_2)\otimes U)=\Hom_G(V,0)=0.
\end{multline*}
\end{proof}

Now we are ready to prove theorem~\ref{th1} for algebraic groups.
The category of representations of the group may not be semisimple here,
so instead of an exceptional collection we get a semiorthogonal decomposition.
\begin{theorem}
\label{th1a}
Let $X$ be an algebraic variety over $\k$, let $G$ be an algebraic group
acting on $X$. Suppose there is a full
exceptional collection $(E_1, \dots, E_n)$ of sheaves in the category $\D(X)$. 
Suppose there is a $G$-equivariant structure on every sheaf $E_i$, 
denote corresponding $G$-sheaf as $\EE_i$.
Then the derived category of $G$-sheaves on $X$
admits a semiorthogonal decomposition:
$$\D^G(X)=\langle\EE_1\otimes\D(Repr(G)),\dots,
\EE_n\otimes\D(Repr(G))\rangle$$
where components are equivalent to the derived category $\D(Repr(G))$
 of finite dimensional representations of the group $G$ 
over $k$.
\end{theorem}
\begin{proof}
We already defined the subcategories $\EE_i\otimes\D(Repr(G))$,
we proved that they are equivalent to the category $\D(Repr(G))$ and 
semiorthogonal. It suffices to check that these categories generate
the whole category $\D^G(X)$. As we have seen, the categories 
$\EE_i\otimes\D(Repr(G))$ are right admissible. By proposition~\ref{bondal},
one needs to show that any object in $\D^G(X)$, right orthogonal to these 
categories, is zero. Indeed, take any $\FF\in\D^G(X),
\FF\ne 0$. The collection $(\EE_1,\dots,\EE_n)$ generates the derived category
$\D(X)$, so $\Hom^p(\EE_i,\FF)\ne 0$ for some numbers $i$ and $p\in\Z$.
From this we deduce that $R\Hom(\EE_i,\FF)\ne 0$. Let 
$V$ denote the object $R\Hom(\EE_i,\FF)$ of the category $\D(Repr(G))$.
The functors $\EE_i\otimes\cdot$ and $R\Hom(\EE_i,\cdot)$ are adjoint,
therefore $$\Hom_G(\EE_i\otimes V,\FF)=\Hom_G(V,R\Hom(\EE_i,\FF))\ne 0,$$
so we get a contradiction.
\end{proof}

\paragraph{The second theorem: case of exceptional collection 
of invariant sheaves in $\D(X)$.}
Now suppose that the action of $G$ preserves 
an exceptional sheaf $E$ on $X$ .
We will understand by this the following  
\begin{definition}
\label{def}
Let us say that 
the action of an algebraic group $G$ on a variety $X$ \emph{preserves} 
a simple (i.e. such that $\End\FF=\k$) sheaf $\FF$ on $X$ if
the sheaf $p_1^*\LL\otimes p_2^*\FF$ on $G\times X$ 
is isomorphic to $a^*\FF$ for some linear bundle $\LL$ on $G$.
In this situation we will also say that
the sheaf $\FF$ is \emph{invariant} under the action of $G$.
\end{definition}

{\bf Remark.} This definition makes sense for all sheaves, not only for
simple. But for sheaves $\FF$, such that $dim_\k \End \FF>1$, it seems to
be incorrect.

Obviously, any equivariant sheaf is invariant. 
Below we will explain the connection between the "naive"
definition of invariant sheaf from section~\ref{fingrp} and
definition~\ref{def}.

It turns out that any invariant exceptional sheaf
is a twisted equivariant sheaf with respect to some
cocycle on $G$.
\begin{predl}
Suppose that the action of a group $G$ on a variety $X$ preserves
an exceptional sheaf $E$ on $X$.
Then $E$ admits an $(\LL,\a)$-equivariant structure for some cocycle 
 $(\LL,\a)$ on $G$.
\end{predl}
\begin{proof}
Fix an isomorphism 
$\theta\colon p_1^*\LL\otimes  p_2^*E\ra a^*E$. There exists
a commutative diagram on
$G\times G\times X$ 
$$\xymatrix{
p_1^*\LL\otimes p_2^*\LL\otimes p_3^*E \ar[rr]^{Id\otimes p_{23}^*\theta} 
\ar[d]_{\a'}&&
p_1^*\LL\otimes (ap_{23})^*E \ar[rr]^{(Id\times a)^*\theta} &&
(a(Id\times a))^*E  \ar@{=}[d]\\
(\mu p_{12})^*\LL\otimes p_3^*E \ar[rrrr]^{(\mu \times Id)^*\theta} &&&&
(a(\mu \times Id))^*E,
}
$$
where 
$$\a'\colon
p_1^*\LL\otimes p_2^*\LL\otimes p_3^*E\ra
(\mu p_{12})^*\LL\otimes p_3^*E 
$$
is a certain isomorphism.
But the sheaf $E$ is exceptional, so $\a'$ is of the 
form $p_{12}^*\a\otimes Id$ for some isomorphism 
$\a\colon  p_1^*\LL\otimes p_2^*\LL\ra\mu^*\LL$ на $G\times G$.
Indeed, we have
\begin{multline*}
\Hom(p_1^*\LL\otimes p_2^*\LL\otimes p_3^*E,                                   
(\mu p_{12})^*\LL\otimes p_3^*E)=
\Hom(p_3^*E,
(p_1^*\LL\otimes p_2^*\LL)^{-1}\otimes(\mu p_{12})^*\LL\otimes p_3^*E)=\\
=\Hom(E,
p_{3*}((p_1^*\LL\otimes p_2^*\LL)^{-1}\otimes(\mu p_{12})^*\LL\otimes p_3^*E))=\\
=\Hom(E,
p_{3*}((p_1^*\LL\otimes p_2^*\LL)^{-1}\otimes(\mu p_{12})^*\LL)\otimes E)=\\
\quad\text{by theorem about flat base change}\\
=\Hom(E,H^0(G\times G,(p_1^*\LL\otimes p_2^*\LL)^{-1}\otimes
\mu^*\LL)\otimes E)=
\Hom(p_1^*\LL\otimes p_2^*\LL,\mu^*\LL).
\end{multline*}

Composing isomorphisms of sheaves on $G\times G\times G\times X$ in
two different ways, we get the associativity condition:
the isomorphisms
$$(Id \times \mu)^*\a\circ (Id\otimes p_{23}^*\a) \quad\text{and}\quad
(\mu \times Id)^*\a\circ (p_{12}^*\a\otimes Id)
$$
between sheaves
$p_1^*\LL\otimes p_2^*\LL\otimes p_3^*\LL$ and $(\mu(Id\times \mu))^*\LL$
on  $G\times G\times G$ are equal.

Hence, the pair $(\LL,\a)$ we got is a cocycle on the group $G$ 
in the sense of definition~\ref{defcoc}. 
\end{proof}

Now we can 
formulate and prove theorem~\ref{th2} for algebraic groups. 
The idea is the same: we produce equivariant sheaves from $E$, tensoring $E$ 
by different  $(\LL,\a)^{-1}$-representations of $G$.

Given a sheaf $E$ as above, we turn it into an $(\LL,\a)$-$G$-sheaf
that would be denoted $\EE$. 
Our aim is to construct an embedding $F_{\EE}$ of the category 
of finite dimensional $(\LL,\a)^{-1}$-representations of $G$
into the category of $G$-sheaves. Define a functor $F_{\EE}$ from
$Repr(G,\LL^*,\a^*)$ to $coh^G(X)$ by
$$F_{\EE}(V)=\EE\otimes V.$$
The sheaf $\EE$ is $(\LL,\a)$-equivariant, and $V$ is an 
$(\LL^*,\a^*)$-representation, therefore by proposition~\ref{lasheves} 
the sheaf 
$\EE\otimes V$ is a (nontwisted) equivariant $G$-sheaf.
The functor $F_{\EE}$ is exact, the derived functor from $F_{\EE}$ will be also 
denoted by $F_{\EE}$.
The proof of the following statement repeats the proof of lemma~\ref{lemma24}.
\begin{lemma}
\label{lemma210}
Functor $F_{\EE}\colon \D(Repr(G,\LL^*,\a^*))\ra\D^G(X)$
is fully faithful. The functors 
$R\Hom(\cdot,\EE)^*$ and  $R\Hom(\EE,\cdot)$ are left and right adjoint 
to $F_{\EE}$ respectively.
\end{lemma}

Denote the image of $\D(Repr(G,\LL^*,\a^*))$ under embedding $F_{\EE}$
by $\EE\otimes \D(Repr(G,\LL^*,\a^*))$.

\begin{theorem}
\label{th2a}
Let $X$ be an algebraic variety over $\k$, let $G$ be an algebraic group
acting on $X$. Suppose there is a full
exceptional collection $(E_1, \dots, E_n)$ of sheaves in the category $\D(X)$. 
Suppose that the action of $G$ preserves every sheaf $E_i$, 
therefore $E_i$ admits an $(\LL_i,\a_i)$-equivariant structure for
some cocycle $(\LL_i,\a_i)$ of the group $G$. Denote corresponding
twisted equivariant sheaf by $\EE_i$. Then the category $\D^G(X)$
admits the semiorthogonal decomposition:
$$\D^G(X)=\left\langle \EE_1\otimes\D(Repr(G,\LL_1^*,\a_1^*)),
\dots, \EE_n\otimes\D(Repr(G,\LL_n^*,\a_n^*))\right\rangle.$$
\end{theorem}

\begin{proof}
The categories  $\EE_i\otimes\D(Repr(G,\LL_i^*,\a_i^*))$ are defined above.
Following the proof of lemma~\ref{lemma26} one can easily check
that these categories are semiorthogonal.
We need to show that categories
$\EE_i\otimes\D(Repr(G,\LL_i^*,\a_i^*))$ generate the category $\D^G(X)$. 
These categories are right admisible by lemma~\ref{lemma210}.
By proposition~\ref{bondal}, it suffices to prove that $\FF=0$ for any object
$\FF$ in $\D^G(X)$, right ortogonal to all categories 
$\EE_i\otimes\D(Repr(G,\LL_i^*,\a_i^*))$.
Assume that $\FF\ne 0$. The collection $(\EE_1,\dots,\EE_n)$  is full,
therefore for some $i$ and $p$ we'll have
$\Hom^p(\EE_i,\FF)\ne 0$. This implies $R\Hom(\EE_i,\FF)\ne 0$.
Denote by $V$ the object $\Hom(\EE_i,\FF)$ of the category 
$\D(Repr(G,\LL_i^*,\a_i^*))$. 
Since functors $\EE_i\otimes\cdot$ and $R\Hom(\EE_i,\cdot)$ are adjoint, 
we get:
$$\Hom_G(\EE_i\otimes V,\FF)=\Hom_G(V,R\Hom(\EE_i,\FF))=\Hom_G(V,V)\ne 0.$$
Hence, $\FF$ is not right orthogonal to 
$\EE_i\otimes\D(Repr(G,\LL_i^*,\a_i^*))$ and we get a contradiction. 
The theorem is proved.
\end{proof}

{\bf Remarks.} 1. Suppose that $G$ is a finite group, field $\k$ is
algebraically closed and  $char(\k)$ doesn't divide the order of $G$.
Then categories of finite dimensional representations of $G$ over $\k$
are semisimple (proposition~\ref{twistedalg}) and generated by
a finite number of orthogonal exceptional objects, namely
irreducible representations of the group. 
We see that theorems~\ref{th1} and~\ref{th2} follow from theorem~\ref{th2a}
in this special case.

2.  In the case $\k=\bar \k, char(\k)=0$ the category 
of finite dimensional representations of a reductive group $G$ 
is semisimple. The same is true for twisted representations for
a cocycle  $(\LL,\a)$ on a a reductive group $G$. Indeed,
the central extension $\~G$ of $G$ by a torus $\k^*$, corresponding
to the cocycle, is also reductive. The category $Repr(G,\LL,\a)$
is a direct summand in the semisimple category $Repr(\~G)$, hence the first
category is also semisimple. 
So, for action of a reductive group $G$, theorem~\ref{th2a}
gives not just a semiorthogonal decomposition of $\D^G(X)$,
but a full exceptional collection in $\D^G(X)$, 
consisting of blocks.

3. We may not require the field $\k$ to be algebraically
closed in  1 and 2. Then categories of representations of the group
would decompose into direct sums of categories of modules over
endomorphism algebras of irreducible representations.
This would result into semiorthogonal decomposition of $\D^G(X)$,
consisting of blocks, with subcategories equivalent to
the derived categories of vector spaces over division algebras over $\k$.

\paragraph{The third theorem: case of exceptional collection 
of invariant blocks of sheaves in $\D(X)$.}
As usual, $X$ denotes a variety and $G$\dash an algebraic group
acting on $X$. Let $H$ be an algebraic subgroup of $G$ of finite index.
Further we will need the notion of \emph{coinduced equivariant sheaf}, 
definition and some properties of coinduced sheaves are presented below:
\begin{predl}
\label{ind}
The forgetful functor $\Res_H^G$  from $coh^G(X)$ to $coh^H(X)$ has
a right adjoint functor, which will be denoted $\Ind_H^G$.
The functor $\Ind_H^G$ is exact.
There exists an exact triple of functors from
$coh^H(X)$ to $coh^H(X)$:
$$0\ra F'\ra \Res_H^G\Ind_H^G\ra Id\ra 0,$$
where $F'$ is some functor. If $G$ is a union of right cosets
modulo $H$, then
$F'(\cdot)=\oplus_{g\in J}g^*(\cdot)$, where $g$ runs through
a set $J$ of representatives in $G(\k)$ of nontrivial right cosets.
\end{predl}

\begin{proof}
Since it is hard to find a good reference, we give a proof.

Consider a commutative diagram of varieties with group actions:
$$\xymatrix{
_HX &  _{G\times H}G\times X \ar[l]_-{p_2} \ar[r]^{\pi}  &  _GG\times ^HX \ar[r]^-{\bar a}
\ar @{=}[d] & _G X.\\
&& _G \bigsqcup_{g\in J}X_g \ar@/^/[ull]^{\bar p=\sqcup Id_X} \ar@/_/[ur]_-{\sqcup g}
}
$$
It gives a  diagram of categories and functors
$$\xymatrix{
coh^H(X) \ar[r]^-{p_2^*}_-{\sim} \ar@/_/[drr]_{\bar p^*}^{\sim}&
  coh^{G\times H}(G\times X)    &  
  coh^G(G\times ^HX) \ar[l]_-{\pi^*}^-{\sim}   \ar@<0.5ex>[rr]^-{\bar a_*=\Ind}
\ar @{=}[d] && coh^G(X) \ar@<0.5ex>[ll]^-{\bar a^*=\Res}.\\
&& coh^G (\bigsqcup X_g)  
}
$$
Comment this diagram.
Variety $G\times X$ is supplied with action of $G\times H$ by formula 
$(g',h)\cdot (g,x)=(g'gh^{-1},hx)$.
Projection $p_2\colon G\times X\ra X$ is a quotient map of the free (left)
action of $G$ on $G\times X$. Since $p_2$ is $H$-equivariant, it  
gives an equivalence between categories $coh^H(X)$ and 
$coh^{G\times H}(G\times X)$.
Map $\pi$ is a projection of $G\times X$ onto it's quotient by the free action 
of $H$.
Since $\pi$ is $G$-equivariant, $\pi^*$ is an equivalence 
$coh^G(G\times ^HX)\ra coh^{G\times H}(G\times X)$.
The map~$\bar a$ is induced by the action morphism $a\colon G\times X\ra X$.
It is $G$-equivariant, hence there are adjoint functors
$\bar a^*\colon coh^G(X)\ra coh^G(G\times ^HX)$
and
$\bar a_*\colon  coh^G(G\times ^HX)\ra coh^G(X)$.
The above equivalence between
$coh^H(X)$ and $coh^G(G\times ^HX)$ sends the functor $\Res_H^G$ 
into $\bar a^*$, and $\bar a_*$ corresponds to a desired $\Ind_H^G$.

Now suppose $G/H$ has a set $J$ of representatives in $G(\k)$.
Then there is an isomorphism (depending on $J$)
$G\times ^HX$ с $\bigsqcup_{g\in J}X_g$, where $X_g=X$.
Therefore we may consider  $\bar a$ to be equal to the 
action of $g$ on the component $X_g$. And for the set of identity
morphisms $\bar p\colon \bigsqcup X_g\ra X$ we'll have $\bar p\pi=p_2$.
We obtain that
$$\Ind_H^G\FF=\bar a_*\bar p^*\FF=\bigoplus_{g\in J}g_*\FF=
\bigoplus_{g\in J^{-1}}g^*\FF.$$
Note that the sum over nontrivial cosets is an $H$-invariant subsheaf,
this implies the statement about $F'$.
\end{proof}

Morally, the forgetful functor $\Res_H^G$
is a pullback for the morphism of stacks 
$X/\!\!/H\ra X/\!\!/G$, and it's right adjoint is a 
pushforward.

The derived functor of the exact functor $\Ind_H^G$ will be denoted
by $\Ind_H^G$.

\begin{theorem}
\label{th3a}
Let  $X$ be an algebraic variety over the field $\k$
and  $G$ be an algebraic group acting on $X$.
Suppose the category $\D(X)$ has a full exceptional collection
of sheaves, consisting of blocks:

\begin{equation}
\label{coll3a}
\left(
\begin{array}{clclc}
E^{(1)}_1 && E^{(2)}_1 && E^{(n)}_1 \\
\vdots & , & \vdots & ,   \dots , & \vdots  \\
E^{(1)}_{k_1} && E^{(2)}_{k_2} && E^{(n)}_{k_n}
\end{array}
\right) ,
\end{equation}
such that the group $G(\k)$ transitively permutes the sheaves in each block. 
Suppose that there exist subgroups $H_i\subset G$ of finite index,
satisfying the foolowing conditions:
the right cosets 
$H_i\setminus G$ have representatives in $G(\k)$,
the sheaf $E^{(i)}_1$ is invariant under the action of $H_i$ and
$H_i(\k)$ is a stabilizer of the sheaf $E^{(i)}_1$ in $G(\k)$.
For some cocycle $(\LL_i,\a_i)$ on the group $H_i$ there is an
equivariant structure on the sheaf 
$E^{(i)}_1$, denote corresponding 
$(\LL_i,\a_i)$-$H_i$-sheaf by $\EE^{(i)}$.
Then the category $\D_G(X)$ admits a semiorthogonal decomposition
\begin{equation}
\label{decomp3a}
D_G(X)=\left\langle \Ind_{H_1}^G(\EE^{(1)}\otimes\D(Repr(H_1,\LL_1^*,\a_1^*))),
\dots, \Ind_{H_n}^G(\EE^{(n)}\otimes\D(Repr(H_n,\LL_n^*,\a_n^*)))\right
\rangle,
\end{equation}
where subcategory 
$\Ind_{H_i}^G(\EE^{(i)}\otimes\D(Repr(H_i,\LL_i^*,\a_i^*)))$
is equivalent to the category 
$\D(Repr(H_i,\LL_i^*,\a_i^*))$.
\end{theorem}

\begin{proof}
Before proof of theorem~\ref{th2a} we constructed embeddings
$F_{\EE_i}\colon \D(Repr(H_i,\LL_i^*,\a_i^*))$ into $\D^{H_i}(X)$, these functors
have left adjoint. We denoted the images of $F_{\EE_i}$ by
$\EE^{(i)}\otimes\D(Repr(H_i,\LL_i^*,\a_i^*))$. Let's
check that the functor $\Ind_{H_i}^G$ is a fully faithful embedding
of the category
$\EE^{(i)}\otimes\D(Repr(H_i,\LL_i^*,\a_i^*))$ into $\D^G(X)$.
Take $V,U\in\D(Repr(H_i,\LL_i^*,\a_i^*))$. Then we have ($i$, $G$ and $H$ 
are omitted for clearer notation):
$$
\Hom_G(\Ind(\EE\otimes V),\Ind(\EE\otimes U))=
\Hom_H(\Res\Ind(\EE\otimes V),\EE\otimes U)=
\Hom_H(\EE\otimes V,\EE\otimes U).
$$
To prove the second equality we use the exact triple
of functors (see proposition~\ref{ind})
$0\ra F' \ra \Res\Ind \ra Id\ra 0$ and vanishing of
$\Hom^i_H(F'(\EE\otimes V),\EE\otimes U)=
\Hom^i_H(\oplus_{g\in J}g^*(\EE\otimes V),\EE\otimes U)$
for $i=-1$ and $0$.
Indeed,
$\Hom_H(\cdot,\cdot)=\Hom_H(\k,\Hom(\cdot,\cdot))$, where
$\k$ stands for trivial representation of the group $H$.
There exists a spectral sequence with 
$E_2^{pq}=\Ext_H^q(\k,\Hom^p(\oplus_{g\in J}g^*(\EE\otimes V),\EE\otimes U))$
that calculates
$E^n=\Hom^n(\oplus_{g\in J}g^*(\EE\otimes V),\EE\otimes U)$.
But $\Hom^p(g^*(\EE\otimes V),\EE\otimes U)=0$ for $g\in J$, because
$\EE$ and $g^*\EE$ are different exceptional sheaves from the same
block in~(\ref{coll3a}) by hypothesis.

Now let's show that subcategories in decomposition~(\ref{decomp3a})
are semiorthogonal.
Suppose that $i<j$, 
$V\in \D(Repr(H_i,\LL_i^*,\a_i^*)), U\in \D(Repr(H_j,\LL_j^*,\a_j^*))$,
then
$$\Hom_G(\Ind_{H_j}^G(\EE^{(j)}\otimes U),\Ind_{H_i}^G(\EE^{(i)}\otimes V))=
\Hom_{H_i}(\Res_{H_i}^G\Ind_{H_j}^G(\EE^{(j)}\otimes U),\EE^{(i)}\otimes V)=0.$$
Indeed,  the spaces
$\Hom^n_{H_i}(\Res_{H_i}^G\Ind_{H_j}^G(\EE^{(j)}\otimes U),\EE^{(i)}\otimes V)$
can be found using the spectral sequence with 
$E_2^{pq}=
\Ext_{H_i}^q(\k,\Hom^p(\Ind_{H_j}^G(\EE^{(j)}\otimes U),\EE^{(i)}\otimes V))$.
But in this sequence
$\Hom^p(\Ind_{H_j}^G(\EE^{(j)}\otimes U),\EE^{(i)}\otimes V)=0$,
because $F'(\EE^{(j)}\otimes U)\cong
(\oplus_{l=2}^{k_j}E^{(j)}_l)\otimes U$, $\EE^{(j)}\otimes U\cong
E^{(j)}_1 \otimes U$ 
(as objects in $\D(X)$) and the sheaves $E^{(j)}_l$ are left
orthogonal to $\EE^{(i)}=E^{(i)}_1$ by hypothesis.

Finally, we check that the triangulated category,
generated by~(\ref{decomp3a}), coincides with $\D^G(X)$. 
As we noticed before, the functor
 $F_{\EE_i}\colon \D(Repr(H_i,\LL_i^*,\a_i^*))\ra \D^{H_i}(X)$ and
the coinducing functor
$\Ind_{H_i}^G\colon \D^{H_i}(X)\ra\D^G(X)$
have left adjoint. Therefore, the subcategories
$\Ind_{H_i}^G(\EE^{(i)}\otimes\D(Repr(H_i,\LL_i^*,\a_i^*)))$
are left admissible. It suffices to check that the left orthogonal
to the right-hand side category in~(\ref{decomp3a}) vanishes.
Take any $\FF\in\D^G(X), \FF\ne 0$. The collection~(\ref{coll3a})  
generates the category $\D(X)$, so for some $i,p$ and $l$
we'll have
$\Hom^p(\FF, E^{(i)}_l)\ne 0$. We may consider the case $l=1$ because $G(\k)$ 
permutes the sheaves transitively in blocks in~(\ref{coll3a}). 
Hence, $R\Hom(\FF,\EE^{(i)})\ne 0$.
Denote by $V$ the object $R\Hom(\FF,\EE^{(i)})^*$ of the category 
$\D(Repr(H_i,\LL_i^*,\a_i^*))$.
Then (because $R\Hom(\cdot,\EE^{(i)})^*$ and $F_{\EE_i}$ are adjoint)
we have:
\begin{multline*}
\Hom_G(\FF,\Ind_{H_i}^G(\EE_i\otimes V))=
\Hom_{H_i}(\FF,\EE_i\otimes V)=\\
=\Hom_{H_i}(R\Hom(\FF,\EE_i)^*,V)=\Hom_{H_i}(V,V)\ne 0.
\end{multline*}
We get a contradiction because $\FF$ is left orthogonal to
the category 
$\Ind_{H_i}^G(\EE^{(i)}\otimes\D(Repr(H_i,\LL_i^*,\a_i^*)))$.
Applying proposition~\ref{bondal} finishes the proof.
\end{proof}

\paragraph{If points of a group preserve an exceptional sheaf
then the sheaf is invariant in the sense of definition~\ref{def}.}

In section~\ref{fingrp} we said that the action of
a finite group $G$ preserves a sheaf $\FF$ if $g^*\FF\cong\FF$
for any $g\in G$. But this definition doesn't work for 
algebraic groups because the group may have few rational points.
Definition~\ref{def} is more reasonable:
a sheaf $\FF$ on $X$ is preserved by the action of a group $G$
if there exists a linear bundle $\LL$ on $G$ such that the sheaves
$p_1^*\LL\otimes p_2^*\FF$ and $a^*\FF$ on $G\times X$ are isomorphic.
Below we relate this condition with invariance of the sheaf under
the action of distinct points of group. Thus we obtain a criterion
for checking definiton~\ref{def}.

\begin{predl}
Let $G$ be an algebraic group acting on a variety $X$. Suppose that $\FF$ is 
an exceptional coherent sheaf on $X$
and for any rational
point $g\in G(\k)$ we have $g^*\FF\cong \FF$. 
Suppose that the above conditions on a sheaf $\FF$ hold
for any finite extension of the field $\k$. 
Then the sheaf $\FF$ is invariant under the action of the group
in the sense of definition~\ref{def}.
\end{predl}

\begin{proof}
Consider the sheaf $\HHom(p_2^*\FF,a^*\FF)$ on $G\times X$.
We claim that the object 
$$Rp_{1*}R\HHom(p_2^*\FF,a^*\FF)$$ of derived category
is  an invertible sheaf on $G$, placed into degree $0$.

Let $g$ be a closed point of the scheme $G$. Consider a Cartesian square
$$
\begin{CD}
X_g @>i>>  G\times X \\
@VVpV  @VV{p_1}V \\
g @>i>> G.
\end{CD}
$$

Since $G\times X$ is smooth, $p_2^*\FF$ is a perfect complex,
and the sheaves $p_2^*\FF$,  $a^*\FF$ are flat over $G$, we have:
\begin{multline*}
Li^*R\HHom(p_2^*\FF,a^*\FF)=\O_{X_g}\otimes^L(p_2^*\FF)^*\otimes^L a^*\FF=
(p_2^*\FF|_{X_g})^*\otimes^L a^*\FF|_{X_g}=\\
=R\HHom(p_2^*\FF|_{X_g},a^*\FF|_{X_g})=R\HHom(\FF',g^*\FF'),
\end{multline*}
where $\FF'$ denotes a sheaf on $X_g$, obtained from $\FF$ by 
scalar extension.
Further, 
$$Rp_*Li^*R\HHom(p_2^*\FF,a^*\FF)=R\Hom(\FF',g^*\FF')=\k(g)[0].$$
Let $\mathrm{H}^{\bullet}$ denote the cohomologies of a complex.
Consider a spectral sequence of vector spaces
$$E_2^{pq}=L_{-q}i^*\mathrm{H}^p(Rp_{1*}R\HHom(p_2^*\FF,a^*\FF))$$
with differential $d_2$ of degree  $(-1,2)$, whose limit is 
$$E_{\infty}^n=\mathrm{H}^n(Li^*Rp_{1*}R\HHom(p_2^*\FF,a^*\FF)),$$
By the flat base change, this equals
$$\mathrm{H}^n(Rp_*Li^*R\HHom(p_2^*\FF,a^*\FF))=
\begin{cases}
\k(g)&\text{when $n=0$,}\\
0&\text{otherwise.}
\end{cases}
$$

Let $\mathrm{H}^p$ be the highest nonzero cohomology of 
$Rp_{1*}R\HHom(p_2^*\FF,a^*\FF)$. If $p>0$, the spectral
sequence implies $E_2^{p0}=E_{\infty}^p=0$,
i.e. $i^*\mathrm{H}^p=0$ for an immersion of any closed point $g$.
Then the sheaf $\mathrm{H}^p$ is zero, a contradiction. Thus $p=0$
and $E_2^{00}=E_{\infty}^0=\k(g)$, the vector space $i^*\mathrm{H}^0$
is one-dimensional for all $g\in G$. This implies that
$\mathrm{H}^0=\LL$ is a linear bundle on $G$.

The functor $p_{1*}$ is left exact, therefore
$$p_{1*}\HHom(p_2^*\FF,a^*\FF)=\mathrm{H}^0(Rp_{1*}R\HHom(p_2^*\FF,a^*\FF))
=\LL.$$
Now consider the composition of homomorphisms
$$p_1^*\LL\otimes p_2^*\FF\ra \HHom(p_2^*\FF,a^*\FF)\otimes p_2^*\FF\ra
a^*\FF,$$
easily, it is an isomorphism.
\end{proof}

\section{Applications.}
Theorems from the previous section can be applyed to different varieties,
in particular, to projective spaces, quadrics, Grassmanians and
del Pezzo surfaces of degree $d\ge 5$.
As a result, we obtain semiorthogonal decompositions
of equivariant derived categories on those varieties with action of
an algebraic group.

\subsection{Projective spaces.}
\label{proj}

Let $V$ be an $n$-dimensional vector space over a field $\k$,
and $\P(V)\cong\P^{n-1}_{\k}$ be it's projectivization.
Consider a group $G$ acting on $\P(V)$.
The category $\D(\P^{n-1})$ has a full exceptional collection
$$(\O,\O(1),\dots,\O(n-1)).$$
Obviously, this collection is preserved by any automorphism of
$\P^{n-1}$, so theorem~\ref{th2a} will work here.                              

The sheaf $\O(-1)$ is a twisted $G$-sheaf with respect to some cocycle
on $G$. We construct this cocycle explicitly.
Consider the fibered product $\~G=G\times_{PGL(V)}GL(V)$, denote
by $\pi$ the projection of $\~G$ on $G$. 
Then $\~G$ is a principal 
$\G_m$-bundle over $G$. Denote the corresponding linear bundle
on $G$ by $\LL$. Multipliciation in $\~G$ gives us a cocycle structure 
on $\LL$, let's call this cocycle $(\LL,\a)$. 
In other words, the group $\~G$ is a central extension of $G$ by~$\G_m$, 
corresponding to cocycle $(\LL,\a)$.

Clearly, the group $\~G$ maps to $G$ and thus acts on $\P(V)$.
On the other hand, there is a tautological representation of 
$\~G$ in $V$, defined by projection of 
$\~G$ into $GL(V)$.
The equivariant $\~G$-bundle
$\O_{\P(V)}\otimes V$ on~$\P(V)$ has an invariant subbundle
$\O_{\P(V)}(-1)$. Therefore the bundle $\O(-1)$ is a $\~G$-bundle
on~$\P(V)$. Note that the action of the subgroup 
$\G_m=\pi^{-1}(e)\subset\~G$ on $\O(-1)$ is linear.
By proposition~\ref{aboutHH},  there is a canonical 
$(\LL,\a)$-representation of $G$ in the vector space $V$ 
and a canonical structure of an $(\LL,\a)$-sheaf on~$\O(-1)$.
Proposition~\ref{lasheves} shows that the bundles $\O(k)$ are
$(\LL,\a)^{-k}$-equivariant bundles on $\P(V)$.    

Alternatively, we may consider the group $\bar G=G\times_{PGL(V)}SL(V)$.
It is a closed subgroup in $\~G$ and it is an extension of 
the group $G$ by the algebraic group $\mu_n$. One can check that
$\bar G$ is a finite extension of $G$ from proposition~\ref{aboutHH1}.

Summing up theorem~\ref{th2a} and propositions~\ref{aboutHH} 
and~\ref{aboutHH1}, we get 

\begin{theorem}
Let $G$ be an algebraic group acting on a projective space
$\P(V)$, let $(\LL,\a)$ be a cocycle on $G$, corresponding
to the action of $G$ on the sheaf $\O(-1)$. 
Then the derived category of $G$-sheaves on $\P(W)$ 
admits the following semiorthogonal:
\begin{multline}
\label{sd_proj}
\D^G(\P(V))=\langle\O\otimes\D(Repr(G)),
\O(1)\otimes\D(Repr(G,\LL,\a)),\dots\\
\dots,
\O(n-1)\otimes\D(Repr(G,\LL^{n-1},\a^{n-1}))\rangle.
\end{multline}
Let 
$\~G=G\times_{PGL(V)}GL(V)$ and $\bar G=G\times_{PGL(V)}SL(V)$.
Then the group $\~G$ is the extension of $G$, given by a cocycle 
$(\LL,\a)$ as in proposition~\ref{aboutHH}. The components
$\D(Repr(G,\LL^i,\a^i))$ in decomposition~(\ref{sd_proj})
are equivalent to the categories $\D(Repr_{(i)}(\~G))$ and
 $\D(Repr_{(i)}(\bar G))$.
\end{theorem}

{\bf Замечание.} It may be interesting to notice that the components
 $\D(Repr_{(0)}(\bar G)),\dots, \D(Repr_{(n-1)}(\bar G))$
of decomposition~(\ref{sd_proj}) are exactly the direct summands of 
the decomposition of $\D(Repr(\bar G))$ into a direct 
sum from proposition~\ref{aboutHH1}.

\medskip
Our decomposition can be viewed as a "noncommutative 
variant" of the semiorthogonal decomposition for relative
Brauer Severi schemes, constructed by M.\,Bernardara in~\cite{Be}.
For a relative Brauer Severi scheme $X\xra{p} S$ of dimension
$n$ over $S$ this decomposition is as follows:
$$\D(X)=\langle\O\otimes p^*\D(S),\O(1)\otimes p^*\D(S,\a^{-1}),
\dots,\O(n)\otimes p^*\D(S,\a^{-n})\rangle.$$
Here $\O(1)$ denotes a relative sheaf $\O_{X/S}(1)$ on $X$.
This sheaf is a twisted sheaf with respect to a cocycle $p^*\a$, where
 $\a\in H^2_{et}(S,\G_m)$ is a certain element in Brauer group.
$\D(S,\a^k)$ denotes a bounded derived category of sheaves on  $S$,
twisted at the cocycle $\a^k$.

We see that decomposition~(\ref{sd_proj}) can be obtained by formal
application of Bernardara's result to the noncommutative
relative Brauer Severi variety $\P(V)/\!\!/G$ over $pt/\!\!/G$.

\subsection{Quadrics.}

Let $\k$ be an algebraically closed field, $char(\k)\ne 2$.
Let $V$ be a vector space over $\k$ of dimension $n, n\ge 3$,
let $Q$ be a nondegenerate quadric in $\P(V)$. Full exceptional
collections in $\D(Q)$ have been constructed by M.\,M.\,Kapranov. 
One of these collections is as follows
\begin{equation}
\label{kapr}
(E_{\pm},\O(-n+3),\dots,\O(-1),\O).
\end{equation}
Here $\O(k)=\O_{\P(V)}(k)|_Q$ denotes the linear bundles,
restricted from $\P(V)$, and $E_{\pm}$ denotes a twisted spinor bundle
$E=\Si(-n+2)$ for odd $n$ and a block of two orthogonal
twisted spinor bundles $E_+=\Si_+(-n+2)$ and $E_-=\Si_-(-n+2)$
for even $n$.
See~\cite{Ka_quad} or~\cite[\S 4]{Ka} for details.

We claim that automorphisms of $Q$ preserve the collection~(\ref{kapr}).
More exactly, automorphisms of $Q$ preserve the bundle $\O(k)$ for any $k$,
send the bundle $E$ into itself (for odd $n$) and send
bundles $E_+$ and $E_-$ into themselves or one into another (for even $n$).

Indeed, take any automorphism $g$ of $Q$. It extends to an automorphism
of the projective space. The sheaves $\O_{\P(V)}(k)$ 
are preserved by automorphisms of $\P(V)$, therefore
for sheaves $\O(-n+3),\dots,\O(-1),\O$ in collection~(\ref{kapr}) we have 
$g^*\O(k)\cong\O(k)$. Further, the subcategory of $\D(Q)$, generated
by $\O(-n+3),\dots,\O(-1),\O$, is $g$-invariant, therefore it's 
right othogonal is also invariant. Since collection~(\ref{kapr})
is full, 
$\langle\O(-n+3),\dots,\O)\rangle^{\perp}=\langle E_{\pm}\rangle$.
If $n$ is odd, the category $\langle E_{\pm}\rangle=\langle E\rangle$
is generated by one exceptional object. All exceptional objects in 
this category are shifts of $E$, so $g^*E\cong E$. 
If $n$ is even, then the category $\langle E_{\pm}\rangle=
\langle E_+,E_-\rangle$ is generated by two exceptional objects
which are orthogonal to each other. All exceptional objects in this category
are shifts either of $E_+$ or of $E_-$.
Therefore, $g^*E_+\cong E_+$ or $E_-$,
$g^*E_+\cong E_-$ or $E_+$.

Now assume that a group $G$ acts on $Q$. The exceptional 
collection~(\ref{kapr})
satisfies hypotheses of theorem~\ref{th3a}, so we can get 
a semiorthogonal decomposition of the category $\D^G(Q)$.

\subsection{Del Pezzo surfaces.}

Let $X$ be a smooth del Pezzo surface of degree $d$
($1\le d \le 9$), and $G$ be a group acting on $X$.
We assume the basic field $\k$ to be algebraically closed 
and to have zero characteristic.
Theorems from previous section provide semiorthogonal decompositions
of derived categories of $G$-sheaves on $X$ in the case $d\ge 5$.

According to a classical result, a smooth del Pezzo surface
is either a result of blowing-up a projective plane in $r=9-d$ 
general points or a smooth quadric (in latter case $d=8$). 
Cases of a projective plane and of a quadric are treated
in previuos sections.
Suppose that $1\le r\le 4$, 
let $X$ be a blow-up of a projective plane in $r$ points, none three of that 
lie on a line. For such $X$ we present a full exceptional collection of 
sheaves, satisfying the hypotheses of theorems~\ref{th3} and~\ref{th3a}.

We will need the following special case of D.\,Orlov's theorem
about blow-ups (see~\cite{Or}):
\begin{theorem}
\label{orlov}
Suppose $\s\colon X\ra \P^2$ is a blow-up of a projective
plane in $r$ distinct points $x_1,\dots,x_r\in \P^2$, and 
$E_i=\s^{-1}(x_i)$ are the exceptional divisors of $\s$.
Then the derived category of coherent sheaves on $X$ has a full
exceptional collection 
\begin{equation}\label{EOOO}
\left(
\begin{matrix}
\O_{E_1}(-1)\\ \vdots\\  \O_{E_r}(-1)
\end{matrix}
\;,
\s^*\O_{\P^2},\s^*\O_{\P^2}(1),
\s^*\O_{\P^2}(2)\right),
\end{equation}
in which the sheaves $\O_{E_1}(-1),\dots,\O_{E_r}(-1)$
form a block (i.e., are orthogonal to each other).
\end{theorem}

We will use notation of theorem~\ref{orlov}.
Let also $L_{ij}$ denote a strict transform of the line $x_ix_j$
under the map $\s$, and $L$ denote a pullback of the divisor class
of a line on $\P^2$.
Note that $\s^*\O_{\P^2}(1)=\O_X(L)$ and $L\sim L_{ij}+E_i+E_j$.
Consider the following cases.

\paragraph{Case $r=1$.}
$X$ is a blow-up of $\P^2$ in the point $x_1$.
The $-1$-curve $E_1$ is unique on $X$, so 
the action of $G$ on $X$ comes from the action of $G$ on
the plane, leaving the point $x_1$ invariant. 
All sheaves of the exceptional collection 
$(\O_{E_1}(-1),\s^*\O_{\P^2},\s^*\O_{\P^2}(1),\s^*\O_{\P^2}(2))$
are preserved by the group, so theorem~\ref{th2a} is applicable here.
Note that $\O_{E_1}(-1)$ is a (usual) $G$-sheaf on $X$, provided by the 
linear action of $G$ on the tangent space $T_{x_1}\P^2$. So 
theorem~\ref{th1a} is also applicable.

In fact, we obtain a semiorthogonal decomposition
$$\D^G(X)=\langle\O_{E_1}(-1)\otimes \D(Repr(G)),\s^*\D^G(\P^2)\rangle.$$

\paragraph{Case $r=2$.} $X$ is a plane with two blown-up points
$x_1$ and $x_2$. There are exactly three $-1$-curves on $X$: 
$E_1, E_2$ and $L_{12}$.
All automorphisms of the surface send a graph of exceptional curves into
itself. Hence all automorphisms of $X$ come from automorphisms of
$\P^2$ preserving the set $\{x_1,x_2\}$.  Thus,
the sheaves
$\s^*\O_{\P^2},\s^*\O_{\P^2}(1),\s^*\O_{\P^2}(2)$ in the 
collection~(\ref{EOOO})
are preserved by the group while the sheaves 
$\O_{E_1}(-1)$ and $\O_{E_2}(-1)$ are preserved or sent into each other. 
Applying theorem~\ref{th3a}, we get a semiorthogonal decomposition
of the category $\D^G(X)$:
$$\D^G(X)=\left\langle
\begin{matrix} \O_{E_1}(-1)\otimes \D(Repr(G))\\
\O_{E_2}(-1)\otimes \D(Repr(G))\end{matrix}\;,\s^*\D^G(\P^2)\right\rangle$$
if the group preserves both points $x_1,x_2\in \P^2$, or
$$\D^G(X)=\left\langle  \Ind_H^G(\O_{E_1}(-1)\otimes \D(Repr(H))),
\s^*\D^G(\P^2)\right\rangle$$
if the group permutes these points.

\paragraph{Case $r=3$.} $X$ is a blow-up of a plane in three points 
$x_1, x_2, x_3$, not lying on a line.

This case is essentially different from the two above cases.
Namely, not all automorphisms of $X$ come from automorphisms of 
the plane and collection~(\ref{EOOO}) is not invariant under 
arbitrary group acting on $X$. But it is possible to obtain an invariant
collection from~(\ref{EOOO}) by mutations.
Consider a collection, consisting of three blocks of linear bundles on $X$:
\begin{equation}
\label{kn3}
\left(\O,\begin{matrix}
\O(L)\\
\O(2L-E_1-E_2-E_3)
\end{matrix},
\begin{matrix}
 \O(2L-E_1-E_2) \\
 \O(2L-E_1-E_3) \\
 \O(2L-E_2-E_3)
\end{matrix}
\right).
\end{equation}
It is a full exceptional collection, it can be obtained from 
collection~(\ref{EOOO}) by mutations.
See the work by 
B.\,Karpov and D.\,Nogin \cite[section 4]{KN} for this and other facts about 
three-block exceptional collections on del Pezzo surfaces.

\begin{lemma}
Blocks $(\O(L),\O(2L-E_1-E_2-E_3))$ and
$(\O(2L-E_1-E_2),\O(2L-E_1-E_3),\O(2L-E_2-E_3))$ are preserved by 
automorphisms of a surface $X$.
\end{lemma}

\begin{proof}
There are exactly six $-1$-curves on $X$. They form a circle 
in the following order: $E_1,L_{12},E_2,L_{23},E_3,L_{13}$. 
This circle has to be invariant under automorphisms of $X$.  
If an automorphism $g$ of $X$ keeps the set of curves $\{E_1,E_2,E_3\}$, 
then it comes from an automorphism of $\P^2$ and sends $L$ into $L$. 
Therefore $g$ leaves invariant each bundle in the block
$(\O(L),\O(2L-E_1-E_2-E_3))$ and permutes bundles in the block
$(\O(2L-E_1-E_2),\O(2L-E_1-E_3),\O(2L-E_2-E_3))$.
To finish the proof it suffices to consider an automorphism $f$
of $X$ switching the two sets of curves
$\{E_1,E_2,E_3\}$ and
$\{L_{12},L_{13},L_{23}\}$. For instance, take as $f$ a
"central symmetry" on a circle of $-1$-curves
(that is, take $f$ such that 
$f(E_1)=L_{23}$,
$f(E_2)=L_{13}$ and so on) 
Such $f$ can be realized as an involution of $X$,
induced by a (?) quadratic transformation of 
$\P^2$ whose centers are $x_1,x_2,x_3$.
Note that $f(L)=2L-E_1-E_2-E_3$ and $f(2L-E_1-E_2-E_3)=L$,
so sheaves from the block 
$(\O(L),\O(2L-E_1-E_2-E_3))$ are permuted by the automorphism $f$.
On the other side, the calculation gives
$$f(2L-E_1-E_2)=2(2L-E_1-E_2-E_3)-(L-E_2-E_3)-(L-E_1-E_3)=2L-E_1-E_2.$$
Therefore (by symmetry), all linear bundles in the block 
$(\O(2L-E_1-E_2),\O(2L-E_1-E_3),\O(2L-E_2-E_3))$ are invariant under $f$.
\end{proof}

We have checked that conditions of theorems~\ref{th3} and \ref{th3a} 
are hold for collection~(\ref{kn3}), so we obtain semiorthogonal
decompositions of equivariant derived categories. 

\paragraph{Case $r=4$.} $X$ is a blow-up of a plane in 
four points, none three of that lie on a line.
Let $K_X$ denote a canonical divisor class on $X$, $K_X=-3L+E_1+E_2+E_3+E_4$.
We will use the following full exceptional collection of sheaves 
on $X$ from Karpov and Nogin's work~\cite[section 4]{KN}:
\begin{equation}
\label{kn4}
\left(
\begin{array}{ccc}
&& \O(L) \\
&& \O(E_1-K_X-L) \\
\O&F&  \O(E_2-K_X-L)\\
&& \O(E_3-K_X-L)\\
&& \O(E_4-K_X-L)
\end{array}
\right).
\end{equation}
Here $F$ is a vector bundle of rank $2$, it can be described by an extension
$$0\ra \O(-K_X-L)\ra F\ra \O(L)\ra 0.$$

We will check below that automorphisms of $X$ send sheaves from the 
right-hand side block of~(\ref{kn4}) into themselves.
Since collection~(\ref{kn4}) is full, this would imply that the bundle $F$ is 
invariant under automorphisms of $X$.

\begin{lemma}
Automorphisms of $X$ permute the bundles 
$$\O(L),\O(E_1-K_X-L), \O(E_2-K_X-L),\O(E_3-K_X-L), \O(E_4-K_X-L).$$
\end{lemma}

\begin{proof}
There are exactly ten $-1$-curves on $X$, namely four $E_i$'s and 
six $L_{ij}$'s.
Since this ten divisors generate the Picard group of $X$, 
it suffices to prove that automorphisms of the graph of $-1$-curves
send bundles from the right-hand size block in~(\ref{kn4})
into each other.
First let us note that the above is true for automorphisms, preserving
the set $\{E_1,E_2,E_3,E_4\}$. Then, let $f$ be an involution,
preserving $E_4$ and
equal to a central symmetry on the circle of six $-1$-curves that
do not intersect $E_4$.
One can easily see that automorphisms, leaving the set $\{E_1,E_2,E_3,E_4\}$
invariant, together with $f$ generate the group of all automorphisms
of the graph.
Hence we need only to check that $f$ leaves the block invariant.
All necessary calculations were done in previous paragraph.
We have:
\begin{gather*}
f(L)=2L-E_1-E_2-E_3=E_4-K_X-L, \\
f(E_4-K_X-L)=L\quad\text{because $f$ is of second order,}\\
f(E_1-K_X-L)=f(2L-E_2-E_3)-E_4=2L-E_2-E_3-E_4=E_1-K_X-L,\\
\text{and similarly}\\
f(E_2-K_X-L)=E_2-K_X-L,\\
f(E_3-K_X-L)=E_3-K_X-L.
\end{gather*}
\end{proof}

We see that collection~(\ref{kn4}) satisfyes conditions of 
theorems~\ref{th3} and~\ref{th3a}.

\subsection{Grassmanians.}                             
Let $V$ be a vector space over a field $\k$ of characteristic $0$, 
$dim(V)=n$,
let $\Gr=\Gr(k,V)$ denote a grassmanian of 
$k$-dimensional subspaces in $V$.

To apply theorems~\ref{th2} and \ref{th2a}, we will need 
an exceptional collection on $\Gr$ that is decribed below.
This collection was constructed by M.\,Kapranov
in~\cite{Ka_grass}, see also~\cite[\S 3]{Ka}. Introduce some notations.
Let $\lam=(\lam_1,\dots,\lam_r)$ be a Young diagram,
whose $r$ rows have lengths $\lam_1,\dots,\lam_r$ 
($\lam_1\ge\lam_2\ge\dots\ge\lam_r$ here denote positive integers).
The total number of cells in $\lam$ we denote by 
$|\lam|=\sum_i \lam_i$. 
Every Young diagram $\lam$ defines a tensor operation on vector spaces
that sends a vector space $V$ into the space denoted by 
$\Sigma^{\lam}V$.  Namely, for a diagram
$\lam$, consisting of $r\le dim(V)$ rows, one can
define $\Si^{\lam}V$ as a space of an irreducible representation 
of the group $GL(V)$ with the highest weight $\lam$ (if necessary, 
zero rows are added to $\lam$). Otherwise $\Si^{\lam}V=0$. 
This space can be described explicitly as some quotient of
natural representation of $GL(V)$ in 
$V^{\otimes|\lam|}$ (see~\cite[chapter 8]{Fu}).
Hence, the correspondence $\Si^{\lam}$
is a covariant functor of $V$ and $\Si^{\lam}$
defines an operation on vector bundles.

Suppose $S$ is a tautological vector bundle on $\Gr=\Gr(k,V)$.
Then, according to Kapranov,
the category $\D(\Gr)$ admits a semiorthogonal decomposition
\begin{equation}
\label{kapr2}
\D(\Gr)=\langle\D_{k(n-k)},\D_{k(n-k)-1},\dots,\D_1,\D_0\rangle.
\end{equation}
The subcategory $\D_i$ is generated by pairwise orthogonal vector
bundles $\Si^{\lam}S$, where $\lam$ runs Young diagrams with $i$
cells having $\le k$ rows and 
$\le n-k$ columns.

Lets see how many automorphisms do grassmanians have.
For $1\le k\le n-1, k\ne n/2$
the canonical map $PGL(V)\ra Aut(\Gr(k,V))$ is an isomorphism, 
but for even $n$ and $k=n/2$ it is not: $PGL(V)$ is a subgroup
in $Aut(\Gr(k,V))$ of index $2$. The following map is an example of
an automorphism of $\Gr(n/2,V)$ not coming from $PGL(V)$:
the map $U\mapsto U^{\perp}$,
where $\perp$ means orthogonal complement with respect to some
nondegenerate quadratic form  on $V$.

Below we'll consider the case of a group action induced by a homomorphism
$G\ra PGL(V)$.
In this case the bundle $S$ and all bundles 
$\Si^{\lam}S$ are preserved by the action and the exceptional
collection~(\ref{kapr2}) satisfies terms of theorems~\ref{th2}~and~\ref{th2a}.

Suppose $(\LL,\a)$ is a cocycle on the group $G$, 
such that there is an $(\LL,\a)$-representation of
$G$ in the space $V$ (and $\O_{\P(V)}(-1)$ is an $(\LL,\a)$-equivariant
subsheaf in $\O_{\P(V)}\otimes V$). 
As in section~\ref{proj},
the cocycle $(\LL,\a)$ corresponds to an extension 
$\~G=G\times_{PGL(V)}GL(V)$
of the group $G$.
Note that $S$ is a $G$-subbundle in $\O_{\Gr}\otimes V$,
therefore $S$ is an $(\LL,\a)$-equivariant bundle,
and $\Si^{\lam}S$ is an $(\LL,\a)^{|\lam|}$-bundle.
Applying theorem~\ref{th2a} and proposition~\ref{aboutHH},
we get the following 
\begin{theorem}
In the above notation the derived equivariant category
$\D^G(\Gr)$ admits a semiorthogonal decomposition
$$\D^G(\Gr)=\langle\D'_{k(n-k)},\D'_{k(n-k)-1},\dots,\D'_1,\D'_0\rangle.$$
Here $\D'_i= \D_i\otimes \D(Repr_{(-i)}(\~G))$ 
denotes a set of pairwise orthogonal subcategories 
$\Si^{\lam}S\otimes  \D(Repr_{(-i)}(\~G))$, where $\lam$ runs
all Young diagrams with $i$ cells having  $\le k$ rows and $\le n-k$ columns.
\end{theorem}

\end{document}